\documentclass[11pt]{amsart}

\usepackage{amsmath}
\usepackage{amsthm}
\usepackage{amsfonts}
\usepackage{amssymb}
\usepackage{amsthm}
\usepackage{bbm}
\usepackage{enumitem}
\usepackage[mathscr]{euscript}
\usepackage{float}
\usepackage{ifthen}
\usepackage{listings}
\usepackage{mathdots}
\usepackage{mleftright}
\usepackage[new]{old-arrows}
\usepackage{rotating}
\usepackage{sansmath}
\usepackage{tabularx}
\usepackage{tikz}
  \usetikzlibrary{positioning}
\usepackage{tikz-cd}
  \tikzcdset{
    cells={font=\everymath\expandafter{\the\everymath\displaystyle}},
  }
\usepackage{todonotes}
\usepackage{xcolor}
\usepackage{xspace}

\usepackage[pdftex,pdfpagelabels,bookmarks,hyperindex,hyperfigures]{hyperref}
  \hypersetup{
    colorlinks,
    linkcolor={red!55!black},
    urlcolor={blue!55!black},
    citecolor={green!55!black}
  }
\usepackage{cleveref}

\definecolor{codegreen}{rgb}{0,0.6,0}
\definecolor{codegray}{rgb}{0.5,0.5,0.5}
\definecolor{codepurple}{rgb}{0.58,0,0.82}
\definecolor{backcolour}{rgb}{0.95,0.95,0.92}
\lstdefinestyle{mystyle}{
    xleftmargin={2.5em},
    backgroundcolor=\color{backcolour},
    commentstyle=\color{codegreen},
    keywordstyle=\color{magenta},
    numberstyle=\tiny\color{codegray},
    stringstyle=\color{codepurple},
    basicstyle=\ttfamily\footnotesize,
    breakatwhitespace=false,
    breaklines=true,
    captionpos=b,
    keepspaces=true,
    numbers=left,
    numbersep=5pt,
    showspaces=false,
    showstringspaces=false,
    showtabs=false,
    tabsize=2
}
\usepackage{subfiles}

\newtheorem{thm}{Theorem}
\numberwithin{thm}{section}

\newtheorem{lem}[thm]{Lemma}
\newtheorem{cnj}[thm]{Conjecture}
\newtheorem{guess}[thm]{Guess}
\newtheorem{cor}[thm]{Corollary}

\theoremstyle{definition}

\newtheorem{conv}[thm]{Convention}
\newtheorem{exm}[thm]{Example}
\newtheorem{defn}[thm]{Definition}

\newtheorem{notn}[thm]{Notation}
\newtheorem{rmk}[thm]{Remark}
\newtheorem{warn}[thm]{Warning}

\newtheorem{asmptn}[thm]{Assumption}

\newcommand{\bb}{\mathbb}

\newcommand{\s}{\mathscr}

\newcommand{\mc}{\mathcal}

\newcommand{\bC}{\bb{C}}

\newcommand{\sC}{\s{C}}

\newcommand{\wh}{\widehat}
\newcommand{\br}[1]{\mleft( #1 \mright)}

\newcommand{\set}[2][]{
  \ifthenelse{\equal{#1}{}}{
    \mleft\{ #2 \mright\}
  }{
    \mleft\{ #1\ :\ #2 \mright\}
  }
}

\renewcommand{\to}[1][]{
  \ifthenelse{\equal{#1}{}}{
    \longrightarrow
  }{
    \stackrel{#1}{\longrightarrow}
  }
}

\newcommand{\To}[1][]{
  \ifthenelse{\equal{#1}{}}{
    \Longrightarrow
  }{
    \stackrel{#1}{\Longrightarrow}
  }
}

\renewcommand{\mapsto}[1][]{
    \ifthenelse{\equal{#1}{}}{
      \longmapsto
    }{
      \stackrel{#1}{\longmapsto}
    }
}

\newcommand{\ot}[1][]{
  \ifthenelse{\equal{#1}{}}{
  \longleftarrow
  }{
    \stackrel{#1}{\longleftarrow}
  }
}

\newcommand{\tr}{\text{tr}\xspace}

\newcommand{\hto}[1][]{\stackrel{#1}{\longhookrightarrow}}

\newcommand{\id}{\mathrm{id}}

\renewcommand{\Ref}[2][]{\ifthenelse{\equal{#1}{}}{\ref{#2}}
                      {\hyperref[#2]{\ref*{#1}(\ref*{#2})}}}
\newcommand{\Aref}[2][]{\ifthenelse{\equal{#1}{}}{\autoref{#2}}
                      {\hyperref[#2]{\autoref*{#1}\ref*{#2}}}}
\newcommand{\Sref}[1]{\hyperref[#1]{\S \ref*{#1}}}
\newcommand{\dom}{\mathrm{dom}}

\newcommand{\Set}{\mathrm{Set}}

\newcommand{\Grpd}{\mathrm{Grpd}}
\newcommand{\op}{\mathrm{op}}

\newcommand{\Fun}{\mathrm{Fun}}

\newcommand{\Hom}{\mathrm{Hom}}
\newcommand{\HHom}{\mathcal{H}\mathrm{om}}
\newcommand{\Map}{\mathrm{Map}}

\DeclareMathOperator{\colim}{\mathrm{colim}}

\newcommand{\Sc}[1][]{\ifthenelse{\equal{#1}{}}{\mathrm{Sch}}{\mathrm{Sch}_{/#1}}}

\newcommand{\Spec}{\mathrm{Spec}}

\newcommand{\Cat}{\s{C}\mathrm{at}}
\newcommand{\Catinf}{\Cat_\infty}

\newcommand{\sSet}{\mathrm{sSet}}

\newcommand{\QCoh}{\mathrm{QCoh}}

\newcommand{\Vect}{\mathrm{Vect}}

\newcommand{\Sym}{\mathrm{Sym}\xspace}
\newcommand{\St}{\mathrm{St}}
\newcommand{\AlgSt}{\mathrm{AlgSt}}

\newcommand{\CRing}{\mathrm{CRing}}
\newcommand{\CAlg}{\mathrm{CAlg}}

\newcommand{\unstr}{\mathrm{unstr}}
\newcommand{\Stab}{\mathrm{Stab}}

\usepackage[style=alphabetic]{biblatex}

\AtNextBibliography{\small}

\title[Moduli Stacks of Quiver Bundles]
      {Moduli Stacks of Quiver Bundles with Applications to Higgs Bundles}

\author{Mahmud Azam}
\email{mahmud.azam@usask.ca}

\author{Steven Rayan}
\email{rayan@math.usask.ca}

\address{Centre for Quantum Topology and Its Applications (quanTA) and Department of Mathematics and Statistics, University of Saskatchewan, SK, Canada~ S7N 5E6}

\begin{document}

\begin{abstract}
We provide a general method for constructing moduli stacks
whose points are diagrams of vector bundles over a fixed base,
indexed by a fixed simplicial set --- that is, quiver bundles of a
fixed shape. We discuss some constraints on
the base for these moduli stacks to be Artin and observe that
a large class of interesting schemes satisfy these constraints. Using this construction,
we recover Nakajima quiver varieties and provide an alternate construction for moduli stacks of Higgs
bundles along with a proof of algebraicity following readily
from the algebraicity of moduli stacks of quiver bundles.
One feature of our approach is that, for each of the moduli stacks
we discuss, there are moduli stacks that are Artin, parametrizing morphisms of
the objects being classified. We discuss some potential applications of this
in categorifying non-abelian Hodge theory in a sense we will make precise.
We also discuss potential applications of our methods and perspectives to the
subjects of quiver varieties, abstract moduli theory, and homotopy theory.
\end{abstract}

\maketitle

\tableofcontents

\section{Introduction}
\label{sec:Intro}

Solving moduli problems is one of the modern techniques for constructing new and interesting spaces. Moduli spaces tend to be very nice objects from a geometric or topological point of view, but the imposition of a stability condition comes at the cost of obscuring some otherwise useful features --- for example, over a Riemann surface or complex curve, the moduli space of Higgs bundles is only birationally equivalent to the cotangent bundle of the moduli space of stable bundles.  On the other hand, moduli stacks are more flexible and tend to have more obvious versions of features like the aforementioned one (so that the moduli stack of Higgs bundles is exactly the tangent bundle to the stack of bundles), but depending on how they are defined they may be cumbersome to interact with. 

In this article, we wish to define, in a relatively straightforward way, a moduli stack whose points are diagrams,
which may or may not commute, of vector bundles over a fixed stack.
Such a moduli stack will be called a \emph{moduli stack of quiver bundles}.  The homological algebra of quiver bundles was developed in \cite{GK:05}, and moduli spaces of these objects were considered very explicitly for low-genus curves in \cite{RS:18,RS:21} using a stability condition induced by viewing these as Higgs bundles of a particular sub-diagonal form, while these were constructed in far greater generality in \cite{AS:05,AS:08,AS:13,AS:17,AS:20,AS:22}, just for example.
Some of the main motivations, in very general terms, of this programme are:
\begin{enumerate}[label=(\roman*), itemsep=0pt]
\item to categorify moduli problems involving vector bundles in a suitable
sense to be expanded upon --- see \cref{sec:ModStQuiver} and \cref{rmk:meaning};
\item to construct and study a categorified stack of Higgs bundles in this sense;
\item to enlargen some existing connections between quiver bundles and Higgs bundles, building on the relationship where the former, at least in Dynkin type $A$, appear as fixed points of an algebraic $\mathbb C^\times$-action on the moduli space of the latter (cf. for example \cite{R18});
\item to recover a notion of mapping space for quiver bundles, in particular
for vector bundles themselves and Higgs bundles, and examine any connections
thereof with homotopy theory;
\item to explore some potential connections between abstract moduli theory and
homotopy theory;
\end{enumerate}
The main goal of this paper is to initiate some minimal foundation for such
investigations.

To accomplish this in as ``model-agnostic'' a way
as possible, we set up a minimal context for defining moduli stacks of quiver
bundles in \cref{sec:Context}, which we will then specialize and build on to study examples
of interest. In doing so, we prioritize a simplicial approach over existing approaches to moduli stacks of, say, Higgs bundles, e.g. \cite[Corollary 2.6(2)]{PTVV13}.
In \cref{sec:ModStVec}, we construct a moduli stack parametrizing triples $(E, F, f)$
of two vector bundles $E, F$ over a fixed base and a vector bundle map $f : E \to F$.
We use this stack to then construct moduli stacks parametrizing diagrams of vector bundles
--- that is, quiver bundles ---
over a fixed base indexed by a fixed simplicial set in \cref{sec:ModStQuiver}. In this section,
we also establish the algebraicity, by which we mean the property of being an
Artin stack, of the stacks constructed in the case of
a somewhat general base $1$--stack over the fppf site of commutative rings and
a finite indexing simplicial set. From this,
we recover the algebraicity of such stacks when the base is a projective variety over a field.
In \cref{sec:QuiverVarieties}, we sketch how to recover quiver varieities and
Nakajima quiver varieties from our framework, the precise details of which we leave
for future work.
In \cref{sec:HiggsBundles}, we provide a construction of a moduli stack of Higgs
bundles, under the conjecture that the moduli stacks of vector bundles and triples admit
suitable morphisms of stacks sending vector bundles and their maps to the respective tensor products.
We observe that the algebraicity of this moduli stack of Higgs bundles is immediate from
the moduli stacks of quiver bundles. Here, we also speculate about a version of non-Abelian Hodge theory involving a stack parametrizing morphisms of Higgs bundles. In \cref{sec:ModHmtpy}, we discuss some potential links between moduli theory and homotopy theory inspired by the perspectives in this work.

\subsubsection*{Conventions}
\label{subsubsec:Conventions}

We will follow the notation and conventions of \cites{HTT, HA, SpAG} for the most part.
In particular, by infinity category we will mean quasicategory and by a space or animum,
we will mean a Kan complex. $\CRing$ will stand for $E_\infty$--algebras in the
quasicategory of spectra. $\CRing^{\heartsuit}$ will stand for the category of
commutative rings which is also the heart of $\CRing$. $\CRing^\Delta$ will stand for
animated rings.
A difference here will be that we will use $\sSet$ to denote the quasicategory
of simplicial sets.
We will use $pr_i : \prod_{k = 1}^n X_k \to X_i$ to denote projections of products
and $\pi_i : \lim_k X_k \to X_i$ to denote projections or structure maps of limits.
We will use $\Delta : X \to X \times X$ to denote diagonal maps. A potential source of
confusion will be our use of $\Delta$ to also denote the simplicial indexing category
but hopefully, this confusion will be easily dispelled by context.
For a quasicategory $\s{C}$, we will write $\Hom_\s{C}(A, B)$ for the Kan complex of
maps $A \to B$ and $\HHom_{\s{C}}(A, B)$ for internal Hom's. We will use the term
``algebraic stack'' to denote Artin $1$--stacks over the fppf site of affine schemes
over a base ring.

Given a category $\s{C}$ and a space or animum $I$,
we will write $\underline{I}$ for the constant presheaf $\s{C}^\op \to \s{S}$
with value $I$ --- that is, the presheaf
$c \mapsto I, \forall c \in \s{C}$.
Whenever we mention limits or colimits, we mean limits in the respective quasicategories
--- for instance, in the case of nerves of $1$--categories, these will be usual limits
and colimits; for $(2, 1)$--categories such as those of $1$--stacks, we will mean
strong $2$--limits and $2$--colimits (that is, $(2, 1)$--limits and
--colimits); and for quasicategories, we will mean the fully
$\infty$--categorical limits and colimits such as those of sheaves of spaces over
an $\infty$--categorical site or $\infty$--stacks. Another potential source of
confusion is the word ``pullback'' which could mean (pre)sheaf pullbacks or fibre products.
However, it will be clear from context which we mean.

Finally, we will make use of the language of internal categories or
category objects in $(2, 1)$--categories --- mainly in the $(2, 1)$--category
of $1$--stacks over some site of affine schemes.
That is, we will consider some ambient $2$--category
$\s{D}$ and an object in $\s{D}$ of ``objects''
$\s{C}_0$ and an object in $\s{D}$ of ``morphisms'' $\s{C}_1$, equipped with
maps $s, t : \s{C}_1 \to \s{C}_0$ to be thought of as sending morphisms
to their sources and targets, a map $e : \s{C}_0 \to \s{C}_1$ to be
thought of as sending an object to its identity morphism, and
a morphism $c : \s{C}_1 \times_{\s{C}_0} \s{C}_1 \to \s{C}_1$,
where the pullback is of the diagram
$\s{C}_1 \to[s] \s{C}_0 \ot[t] \s{C}_0$, to be thought of as sending
a composeable pair of morphisms to their composite. These maps
satisfy diagrams that express the associativity and unitality of
composition. Of course, these are diagrams in a $(2, 1)$--category.\\

\noindent\textbf{Acknowledgements.} The authors thank Kuntal Banerjee, Daniel Halpern-Leistner, Matthew Koban, Kobi Kremnizer, Carlos Simpson, Motohico Mulase, and Qixiang Wang for very useful discussions.  We acknowledge Dat Minh Ha for valuable discussions around the formulation of Theorem \eqref{thm:algebraicity}. The authors are especially grateful to Toni Annala for hosting the first-named author for a crucial research stay at the Institute for Advanced Study in March 2024 and for many discussions after that time, leading to a number of crucial developments around the parametrization of vector bundle triples in this work.  We are also indebted to Antoine Bourget for hosting the first-named author for a two-month research stay during the months of October and November 2024 at CEA Saclay, during which the major part of the section on quiver varieties was completed, and for many discussions about quiver varieties. We are pleased to acknowledge that this stay was funded by a High Level Scientific Fellowship (HLSF) for graduate students from Campus France and the Embassy of France in Canada. Formative moments towards the completion of the work occurred during the Arithmetic Quantum Field Theory Conference at Harvard in March 2024, in which the first-named author participated through a generous conference fellowship from the Arithmetic Quantum Field Theory Program; the Workshop on Advances in Higgs Bundles at the Brin Mathematics Research Centre in Maryland in April 2024 in which the second-named author was a speaker and participant; the 2nd Simons Math Summer Workshop, convened around the topic of ``Moduli'' at the Simons Center for Geometry and Physics in July 2024, in which the first-named author was a participant in residence, once again with generous student funding; and finally the Workshop on Hamiltonian Geometry and Quantization at the Fields Institute in July 2024, in which the second-named author was a speaker and participant. The authors thank the various organizing teams of these gatherings for creating such stimulating research environments. The second-named author was supported during the development of this work by a Natural Sciences and Engineering Research Council of Canada (NSERC) Discovery Grant, a Canadian Tri-Agency New Frontiers in Research Fund (NFRF) Exploration Stream Grant, and a Pacific Institute for the Mathematical Sciences (PIMS) Collaborative Research Group Award, for which this manuscript is PIMS report PIMS-20240716-CRG34. The first-named author was funded by a University Saskatchewan Dean's Scholarship, an NSERC Canada Graduate Scholarship (Doctoral), and the second-named author's Discovery Grant.  We gratefully acknowledge these funding sources.

\section{Context}
\label{sec:Context}

Let $\s{C}$ be a quasicategory to be thought of as the quasicategory
of local models for a geometric theory --- for instance, affine schemes
$\s{C} = \CAlg^{\heartsuit, \op}$, or spectral affine schemes
$\s{C} = \CAlg^\op$, or derived affine schemes
$\s{C} = \CAlg^{\Delta, \op}$.
We assume $\s{C}$ is complete and cocomplete. In particular, it has
fibre products, or dually, tensor products of commutative algebras.
Let $\tau$ be a subcanonical topology on $\s{C}$. Consider sheaves
of spaces or anima on $\s{C}$ with respect to $\tau$ and denote the
quasicategory of these objects $\St_{\s{C}, \tau}$, omitting $\tau$ from
notation when there is no confusion --- these are to be thought of as
stacks, spectral stacks or derived stacks.
Notice that we do not yet concern ourselves with any kind of algebraicity.
We also adopt a notational convention for convenience: we will write
$\St_{\s{C}}^n$ for the the full subcategory of $n$--truncated objects
--- that is, sheaves of $(n, 1)$--groupoids on $\s{C}$. We will take
$n = \infty$ to mean $\St_{\s{C}}^\infty = \St_\s{C}$.
We will mostly be concerned with $n = 1$ in this paper,
but will make our definitions and constructions for $n = \infty$
whenever possible, as we intend to eventually extend our work in this paper
to most available models of geometry including derived algebraic geometry,
spectral algebraic geometry, and derived analytic geometry with modifications
as necessary.

We will then consider the following assumptions on
$\St_{\s{C}}^n$ for all $n$:

\begin{asmptn}
For each $X \in \St^n_\s{C}$, there exists a symmetric monoidal stable
quasicategory $\QCoh(X)$ equipped with a $t$-structure and admitting
small colimits.
We further assume that $\QCoh(X)$ is symmetric monoidal with
the product preserving colimits in each variable --- that is,
the mondoidal product is bilinear. We will always assume a choice of
unit object for the monoidal product and denote it $\mc{O}_X$.
As the notation suggests, $\QCoh(X)$ is to be
thought of as the quasicategory of chain complexes of quasicoherent
sheaves on $X$ --- this has a bilinear monoidal product by
\cite[Proposition 2.1.0.3, Proposition 2.2.4.1, Proposition 2.2.4.2]{SpAG}.
\end{asmptn}

\begin{asmptn}
$\QCoh(X)$ has a full subcategory $\Vect(X)$ closed
under the monoidal product and containing the unit of the
monoidal product, and hence inheriting
a symmetric monoidal structure from $\QCoh(X)$. Of course, $\Vect(X)$ is to be
thought of as the full subcategory of finite locally free sheaves.
\end{asmptn}

\begin{notn}
In the case that $\s{C} = \CRing^{\heartsuit, \op}_{/R}, n = 1$ for some ring $R$,
we will also write
$\QCoh(X)$ and $\Vect(X)$ to denote the full-subcategories of $1$--truncated
objects --- that is, the usual categories of quasicoherent sheaves or finite locally
free sheaves respectively.
\end{notn}

\begin{asmptn}
$\Vect(X)$ is closed under coproducts in $\QCoh(X)$, and these
coproducts coincide with products in $\Vect(X)$. We will denote the product/coproduct by
$- \oplus -$.
\end{asmptn}

\begin{asmptn}
There is a functor $(-)^\vee : \Vect(X) \to \Vect(X)$ such
$E^\vee \otimes F \simeq \HHom_{\Vect(X)}(E, F)$, the internal $\Hom$,
for each $E, F \in \Vect(X)$
and we have a closed monoidal structure on $\Vect(X)$ given by the adjunction
$- \otimes E \dashv E^\vee \otimes - \simeq \HHom_{\Vect(X)}(E, -)$.
\end{asmptn}

\begin{asmptn}
The assignment $\QCoh$ extends to a sheaf of quasicategories:
\[\begin{array}{ccccc}
\QCoh &:& \s{C}^\op &\to& \Catinf \\
      &:& Y &\mapsto& \QCoh(Y)\\
      &:& (f : Y \to Z) &\mapsto& (f^* : \QCoh(Z) \to \QCoh(Y))
\end{array}\]
which restricts to a sheaf of categories:
\[\begin{array}{ccccc}
\Vect &:& \s{C}^\op &\to& \Catinf \\
      &:& Y &\mapsto& \Vect(Y)\\
      &:& (f : Y \to Z) &\mapsto& (f^* : \Vect(Z) \to \Vect(Y))
\end{array}\]
We will write $\QCoh$ and $\Vect$ to also
denote the left Kan extensions of the above functors along the Yoneda
embedding $\s{C} \hto \St^n_\s{C}$.
\end{asmptn}

\begin{asmptn}
For any $f : X \to Y$, the pullback functor $f^*$ is symmetric monoidal and
preserves duality: $f^*(E^\vee) \simeq (f^*E)^\vee$ for each
$E \in \Vect(X)$. 
\end{asmptn}

\begin{asmptn}
For each $F\in \Vect(X)$, there is a map
$\tr : F \otimes F^\vee \to \mc{O}_X$ such that for each
$E, G \in \Vect(X)$, the map
\[
\id \otimes \tr \otimes \id : E^\vee \otimes F \otimes F^\vee \otimes G
\to E^\vee \otimes \mc{O}_X \otimes G \simeq E^\vee \otimes G
\]
is the internal composition map of
$\Vect(X)$ for its closed monoidal structure.
\end{asmptn}


\begin{asmptn}
There exists a functor
$\s{V}_X : \QCoh(X) \to \St^n_{\s{C}/X}$ such that
$\Vect(X) \hto \QCoh(X) \to[\mc{V}_X] \St^n_{\s{C}/X}$ is
faithful.
In addition, the functor $\mc{V}_X$ restricted to $\Vect(X)$ is functorial in $X$,
in the sense that for every morphism $f : X \to Y$,
the functor $f^* : \Vect(Y) \to \Vect(X)$
makes the following diagram of quasicategories commute:
\[\begin{tikzcd}
\Vect(Y) \ar[r, "f^*"] \ar[d, "\s{V}_Y" left] &
\Vect(X) \ar[d, "\s{V}_X"] \\
\St^n_{\s{C}/Y} \ar[r, "f^*" below] & \St^n_{\s{C}/X}
\end{tikzcd}\]
where the bottom horizontal arrow is given by taking pullbacks along $f$.
This is to be thought of as sending a
quasicoherent sheaf to the relative spectrum of its symmetric algebra.
For any $Y, Z \in \St_{\s{C}/X}^n$ with
$Y \simeq \mc{V}_X(E), Z \simeq \mc{V}_X(F)$ for $E, F \in \Vect(X)$, we will
define $Y \otimes Z := \mc{V}_X(E \otimes F), Y^\vee := \mc{V}_X(E^\vee)$.
\end{asmptn}

\begin{defn}
We define $\mc{M}_{\Vect(X)}$ as the mapping stack
$\Map(X, \Vect)$ so that maps $f : U \to \mc{M}_{\Vect(X)}$ are
maps $U \times X \to \Vect$ and hence correspond to objects of
the unstraightening of $\Vect$ over $U \times X$. In the concrete examples,
this would make $\mc{M}_{\Vect(X)}$ the moduli stack of vector bundles
over $X$ --- see \cite[Definition 10.7 and Construction 9.1]{AK22} for
details --- in particular for
$\St_{\CRing^{\heartsuit, \op}}^1$, $\mc{M}_{\Vect}$ is precisely
the moduli stack of vector bundles over $X$.
\end{defn}

\begin{asmptn}
For each $X \in \St^n_\s{C}$, there exists a map
$p_X : \mc{E}_X \to \mc{M}_{\Vect(X)} \times X$, such that for every
$U \in \St^n_\s{C}$ and every $E \in \Vect(U \times X)$,
there exists a map $f_E : U \to \mc{M}_{\Vect(X)}$ and an equivalence:
\[
(f_E \times \id_X)^*\mc{E}_X \simeq \s{V}_X(E)
\]
This is to be thought of as $\mc{M}_{\Vect(X)}$ being a fine moduli stack of
vector bundles over $X$ with a universal
vector bundle $\mc{E}_X$. We will call them as such.
\end{asmptn}

\begin{conv}
In the case that $\s{C} = \CRing^{\heartsuit, \op}, n = 1$, we will take
$\mc{E}_X$ to mean universal vector bundle
\end{conv}

\begin{thm}
The $(2, 1)$--category $\St_{\CRing^{\heartsuit, \op}, \mathrm{fppf}}^1$
of $1$--stacks on the site of commutative rings with the fppf topology
satisfies the above assumptions.
\end{thm}
\begin{proof}
Omitted, as these are well known properties.
\end{proof}

This is the case we will be most interested in for the present work. Nevertheless,
we will make the following conjectures, to be proved in future work:

\begin{cnj}
For $n = \infty$, if we take $\s{C}$ to be the opposite of animated commutative rings
and $\St_\s{C}$ to be derived stacks, then
these assumptions are satisfied if we take
$\mc{V}_X := \underline{\Spec}_X\mathrm{L}\Sym^*(-)$.
\end{cnj}

\begin{rmk}
For spectral algebraic geometry, one needs exercise caution, for the relative spectrum
of the symmetric algebra of a locally free sheaf might not be flat.
\end{rmk}

\begin{cnj}
The models of geometry in condensed mathematics such as analytic stacks
\cite{CS19},
and the models of geometry based on bornological rings \cite{BKK24}
satisfy these assumptions.
\end{cnj}

Given any setting of geometry where the above assumptions are
satisfied, we will make the following basic definitions
in that setting:

\begin{defn}[Quiver]
We will use the term quiver for a simplicial set $I$ for keeping
connection with the theory of quiver varieties.
If $I$ is a standard simplex $\Delta^n$, we will call
quivers of shape $I$, standard $n$--quivers.
\end{defn}

\begin{exm}
The following is a drawing of quiver in our context:
\[\begin{tikzcd}
a \ar[r] \ar[d] \ar[rd] & b \ar[d] \\
c \ar[r] \ar[ru, phantom, "\times" near start] & d
\end{tikzcd}\]
where the ``$\times$'' indicates that we do not have a $2$--simplex witnessing the
commutativity of the bottom left triangle. That is, quivers in this paper
may or may not have composites of composeable arrows. In usual quiver theory,
existence of composites is never assumed.
\end{exm}

\begin{warn}
All our quivers are unlabelled as we will not need any
labellings for the results of this paper.
\end{warn}

\begin{defn}[Quiver Bundle]
A quiver bundle over $X \in \St^n_\s{C}$ is just a map of simplicial sets
$f : I \to \Vect(X)$.
If $I$ is a standard $n$--quiver,
then $f$ is called a standard $n$--quiver bundle.
\end{defn}

\begin{conv}
Unless specified otherwise, in all discussions involving quiver bundles
we will always assume a fixed stack $X \in \St_{\s{C}}$ over which we
consider quiver bundles.
\end{conv}

\section{Moduli Stack of Vector Bundle Morphisms}
\label{sec:ModStVec}

As a foundation, we first define a stack whose points are triples
consisting of two vector bundles and a morphism between them.
Let $X \in \St_\s{C}$ be a stack over $\s{C}$.
For notational convenience, we will write
$B_X := \mc{M}_{\Vect(X)} \times X$ so that we have the universal
vector bundle $p_X : \mc{E}_X \to B_X$, where we are identifying
$\mc{E}_X$ with $\mc{V}_X(\mc{E}_X)$.
We will write $B_X \ot[pr_1] B_X \times B_X \to[pr_2] B_X$ for the two
projections, and $\mc{E}_{X, i} = pr_i^*\mc{E}_X$. We notice
that by our assumptions $\mc{E}_{X, i}^\vee \simeq pr_i^*(\mc{E}_X^\vee)$.
We then consider the vector bundle
$\mc{E}_X^\vee \boxtimes \mc{E}_X
= pr_1^*\mc{E}_X^\vee \otimes_{\mc{O}_{B_X \times B_X}} pr_2^*\mc{E}_X$.
A map $f : U \times X \to \mc{E}_X^\vee \boxtimes \mc{E}_X$ yields
four maps $f_1, f_3 : U \times X \to \mc{M}_{\Vect(X)},
f_2, f_4 : U \times X \to X$, by projection. Using this,
we can take a pullback
\[\begin{tikzcd}
(f_1 \times \id_X)^*\mc{E}^\vee
\otimes (f_3 \times \id_X)^*\mc{E} \ar[r] \ar[d]
    \ar[rd, phantom, "\lrcorner" very near start] &
\mc{E}_X^\vee \boxtimes \mc{E}_X \ar[d] \\
U \times X \ar[r, "{(f_1 \times \id, f_3 \times \id)}" below] &
B_X \times B_X
\end{tikzcd}\]

Now, if $f_2$ and $f_4$ were the projections on to the $X$ components, then
the map $f$ gives a section of the pulled back bundle. In this situation,
we will say the map $f$ is over $X \times X$. The pulled back
bundle is $\HHom_{\Vect(X)}(E, F)$, where $E = (f_1 \times \id)^*\mc{E}_X$
and $F = (f_3 \times \id)^*\mc{E}_X$. In the case of $1$--stacks
on $\CRing^{\heartsuit, \op}$ --- that is, for $\St_\s{C}^1$,
this would yield a vector bundle map
$\tilde{f} : E \to F$ over $U \times X$ --- in fact, vector bundle maps
$E \to F$ are precisely global sections of $\HHom_{\Vect(X)}(E, F)$ which are precisely
maps of the form $f$ which, composed with the projection
$\mc{E}_X^\vee \boxtimes \mc{E}_X \to X$, give the projections
$U \times X \to X$. Since every vector bundle over $U \times X$ corresponds
to a map $U \to \mc{M}_{\Vect(X)}$, we have a bijection:
\[\begin{tikzcd}
\set{(U \times X)\text{--points of } \mc{E}_X^\vee \boxtimes \mc{E}_X
\text{ over } X \times X} \ar[d] \\
\set{\text{vector bundle maps over } U \times X}
\end{tikzcd}\]
The domain of this map is the object set of a category but the codomain
is just a set. However, we can see that the domain is a discrete category and
hence a set. We observe that $\mc{E}^\vee_X \boxtimes \mc{E}_X$ is a stack
fibred/valued in $\Set$ for it is the relative spectrum of the symmetric
algebra of an $\mc{O}_{B_X \times B_X}$--module, which is a sheaf
of sets and not non-discrete groupoids. The relative spectrum by construction
produces a sheaf of sets an $\mc{O}_{B_X \times B_X}$--module. Then, by
the Yoneda lemma, the $(U \times X)$--points of
$\mc{E}_X^\vee \boxtimes \mc{E}_X$ is precisely the fibre
$(\mc{E}_X^\vee \boxtimes \mc{E}_X)(U \times X)$ which is a set.

We then observe that we have the following correspondence by the
Cartesian closed structure of stacks:
\[\begin{tikzcd}
\set{(U \times X)\text{--points of } \mc{E}_X^\vee \boxtimes \mc{E}_X
\text{ over } X \times X} \ar[d, "\simeq"] \\
\set{\text{\parbox{0.8\textwidth}{\centering$U$--points of
$\Map(X, \mc{E}_X^\vee \boxtimes \mc{E}_X)$ over the point in
$\Map(X, X \times X)$ corresponding to the diagonal map
$X \to X \times X$}}}
\end{tikzcd}\]
That is, we are considering $U$--points of the stack define as follows:

\begin{defn}[Moduli Stack of Triples/Vector Bundle Morphisms]
We define a stack $\mc{M}_{\Vect(X), 1}$ as the following pullback of stacks:
\[\begin{tikzcd}
\mc{M}_{\Vect(X), 1} \ar[r] \ar[d]
    \ar[rd, phantom, "\lrcorner" very near start] &
\Map(X, \mc{E}_X^\vee \boxtimes \mc{E}_X) \ar[d] \\
\mathrm{pt} \ar[r, "\overline{\Delta}" below] &
\Map(X, X \times X)
\end{tikzcd}\]
where the right vertical map is the one given by post composition
by the projection
$\mc{E}_X^\vee \boxtimes \mc{E}_X \to B_X \times B_X \to X \times X$,
and the map $\overline\Delta$ is the one corresponding to the
diagonal map $\Delta : X \to X \times X$ under the equivalence of
stacks:
\[
\Map(\mathrm{pt}, \Map(X, X \times X)) \simeq \Map(\mathrm{pt} \times X, X \times X)
\simeq \Map(X, X \times X)
\]
We will call $\mc{M}_{\Vect(X), 1}$ the moduli stack of triples
or the moduli stack of vector bundle morphisms.
\end{defn}


We have the following statement, by construction:
\begin{lem}\label{lem:M1_points}
When $\s{C} = \CRing^{\heartsuit, \op}$ and we consider all the above
constructions for $1$--stacks --- that is, in $\St_\s{C}^1$ ---
$U$--points of $\mc{M}_{\Vect(X), 1}$ correspond to triples
$(E, F, f)$, where $E, F$ are vector bundles over $U \times X$ and
$f$ is a map of vector bundles $E \to F$.
\end{lem}

The stack $\mc{M}_{\Vect(X), 1}$ is equipped with two maps to
$\mc{M}_{\Vect(X)}$, defined as follows:
\begin{align*}
& \mc{M}_{\Vect(X), 1} \\
\to & \Map(X, \mc{E}_X^\vee \boxtimes \mc{E}_X)
    && \text{projection of pullback} \\
\to & \Map(X, B_X \times B_X)
&& \text{vector bundle projection} \\
\to & \Map(X, B_X) && \text{projection of product} \\
\to & \Map(X, \mc{M}_{\Vect(X)})
    && \text{projection } B_X = \mc{M}_{\Vect(X)} \times X
        \to \mc{M}_{\Vect(X)} \\
\to[\simeq] & \Map(X, \Map(X, \Vect)) \\
\to[\simeq] & \Map(X \times X, \Vect) \\
\to & \Map(X, \Vect) && \text{diagonal } X \to X \times X \\
\to[\simeq] & \mc{M}_{\Vect(X)}
\end{align*}
where the equivalences are the immediate ones --- either by definition or
by the Cartesian closed structure of stacks.
The two choices of projection $B_X \times B_X \to B_X$ above gives
two different maps
\[
s, t : \mc{M}_{\Vect(X), 1} \to \mc{M}_{\Vect(X)}
\]

Suppose again that we are in the case of $1$--stacks over
$\CRing^{\heartsuit, \op}$. Consider a $U$--point
$f' : U \to \mc{M}_{\Vect(X), 1}$ which corresponds to a
$f : U \times X \to \mc{E}_X^\vee \boxtimes \mc{E}_X$ such that
its composite with the vector bundle projection is a map of the form
$(f_1 \times \id_X, f_3 \times \id_X) : U \times X \to B_X \times B_X$.
The parts of the maps $s$ and $t$ up to
$\Map(X, \mc{M}_{\Vect(X)})$ composed with $f'$ give maps
$f_1', f_3' : U \to \Map(X, \mc{M}_{\Vect(X)})$ corresponding to
\[
U \times X \to[f_i \times \id_X] \mc{M}_{\Vect(X)} \times X \to
\mc{M}_{\Vect(X)}
\]
for $i = 1, 3$ respectively, where the last map is the projection.
These, in turn, correspond to two maps
\[
U \times X \times X \to[f_i'' \times \id_X]
\Vect \times X \to
\Vect
\]
Composing with $U \times X \to[\id_U \times \Delta] U \times X \times X$,
we get the map
\[
U \times X \to[\id_U \times \Delta]
U \times X \times X \to[f_i'' \times \id_X]
\Vect \times X \to
\Vect
\]
which is simply
$U \times X \to[f_i'' \times \id_X] \Vect \times X \to \Vect$.
This is precisely the $U$--point of $\Map(X, \Vect) = \mc{M}_{\Vect(X)}$
given by $f_i : U \to \mc{M}_{\Vect(X)}$. Thus, $s \circ f$ gives
the domain or source of the vector bundle map classified by $f$ and
$t \circ f$ gives the target or codomain.

It is then natural to seek a morphism
$e : \mc{M}_{\Vect(X)} \to \mc{M}_{\Vect(X), 1}$ that sends a point of
$\mc{M}_{\Vect(X)}$ to the identity morphism of the vector bundle classified by that point.
A map of the form $e$ is equivalent to a map
$e' : \mc{M}_{\Vect(X)} \times X \to \mc{E}_X^\vee \boxtimes \mc{E}_X$ such that
$e'$ post-composed with the projection
$\mc{E}_X^\vee \boxtimes \mc{E}_X \to B_X \times B_X \to X \times X$
factors as $\mc{M}_{\Vect(X)} \times X \to X \to[\Delta] X \times X$ where the first map
is the projection. We would like this map $e'$ to behave like the identity morphism
$\mc{E}_X \to \mc{E}_X$, in some sense, not yet precise. However, we can find a section
$e'' : \mc{M}_{\Vect(X)} \times X \to \mc{E}_X^\vee \otimes \mc{E}_X
\simeq \HHom_{\mc{O}_{B_X}}(\mc{E}_X, \mc{E}_X)$ corresponding
to the identity map $\mc{E}_X \to \mc{E}_X$. We then use the basic fact that
$\id_{B_X} = pr_1 \circ \Delta_{B_X} = pr_2 \circ \Delta_{B_X}$,
where $\Delta : B_X \to B_X \times B_X$ is the diagonal, to see that
$\Delta_{B_X}^*pr_i^*\mc{E}_X = \mc{E}_X$. That is,
$\Delta_{B_X}^*(\mc{E}_X^\vee \boxtimes \mc{E}_X) = \mc{E}_X^\vee \otimes \mc{E}_X$.
Thus, $e''$ yields a map
$B_X = \mc{M}_{\Vect(X)} \times X \to \mc{E}_X^\vee \boxtimes \mc{E}_X$ which we take
to be our $e'$. It is then, straightforward to see that the map
$\mc{M}_{\Vect(X)} \times X \to[e'] \mc{E}_X^\vee \boxtimes \mc{E}_X \to X \times X$
factors as $\mc{M}_{\Vect(X)} \times X \to X \to[\Delta] X \times X$, giving our map
$e : \mc{M}_{\Vect(X)} \to \mc{M}_{\Vect(X), 1}$. In the case
$\s{C} = \CRing^{\heartsuit, \op}$ and $1$--stacks, it is also straightforward to verify,
using arguments similar to the ones used to examine $s, t$, that for any
$U$--point $f : U \to \mc{M}_{\Vect(X)}$, $f \circ e$ corresponds to the triple
$(E, E, \id_X)$, where $E$ is the vector bundle classified by $f$.

Next, in the case $\s{C} = \CRing^{\heartsuit, \op}$ and $1$--stacks again,
it is natural to have a map
\[
c : \mc{M}_{\Vect(X), 1} \times_{\mc{M}_{\Vect(X)}} \mc{M}_{\Vect(X), 1}
\to \mc{M}_{\Vect(X), 1}
\]
where the pullback is taken for the two maps $s, t$,
that sends a point of its domain to the composite of the two
triples classified by the two projections of the point to the two copies of
$\mc{M}_{\Vect(X), 1}$. We take inspiration from the case of vector spaces. For three
vector spaces $U, V, W$ over a field $k$, the composition map
$\Hom(V, W) \otimes \Hom(U, V) \to \Hom(U, W)$ is given by
\[
V^* \otimes W \otimes U^* \otimes V \to[\simeq] V^* \otimes V \otimes U \otimes W
\to[\tr \otimes \id] k \otimes U^* \otimes W \to[\simeq] U^* \otimes W
\]
where $\tr$ is the trace map given by the counit of the Hom-tensor adjunction.
We notice that we do have the counit $E^\vee \otimes E \to \mc{O}_{Y}$ for any vector
bundle $E$ over a stack $Y$, by the Hom-tensor adjunction in $\Vect(X)$.
Thus, if we can map $\mc{M}_{\Vect(X), 1} \times_{\mc{M}_{\Vect(X)}} \mc{M}_{\Vect(X), 1}$
to $\mc{E}_X^\vee \boxtimes \mc{E}_X \boxtimes \mc{E}_X^\vee \boxtimes \mc{E}_X$, we can try
taking some sort of trace for the two middle factors.

To this end, we consider the projections of the pullback
\[
\mc{M}_{\Vect(X), 1} \times_{\mc{M}_{\Vect(X)}} \mc{M}_{\Vect(X), 1}
\to[\pi_i] \mc{M}_{\Vect(X), 1}, i = 1, 2
\]
Each copy of $\mc{M}_{\Vect(X), 1}$ has the canonical
map to $\Map(X, \mc{E}_X^\vee \boxtimes \mc{E}_X)$. This gives two maps:
\[
\mc{M}_{\Vect(X), 1} \times_{\mc{M}_{\Vect(X)}} \mc{M}_{\Vect(X), 1}
\to \Map(X, \mc{E}_X^\vee \boxtimes \mc{E}_X), i = 1, 2
\]
Each of these maps, in turn, correspond to two maps
\[
\mc{M}_{\Vect(X), 1} \times_{\mc{M}_{\Vect(X)}} \mc{M}_{\Vect(X), 1} \times X
\to[\pi_i \times \id_X] \mc{M}_{\Vect(X), 1} \times X
\to \mc{E}_X^\vee \boxtimes \mc{E}_X
\]
for $i = 1, 2$. It is straightforward to verify that
these maps make the following diagram commute, where we write $\mc{M}_1$ and
$\mc{M}_0$ for $\mc{M}_{\Vect(X), 1}$ and $\mc{M}_{\Vect(X)}$ respectively for
brevity:
\begin{equation}\label{eqn:stpullback_proj}
\begin{tikzcd}
\mc{M}_{1} \times_{\mc{M}_{0}} \mc{M}_{1} \times X
    \ar[d, "{(\pi_1, \pi_2)} \times \id_X" left]
    \ar[r, "\pi_i \times \id_X"] &
\mc{M}_{1} \times X \ar[r] &
\mc{E}_X^\vee \boxtimes \mc{E}_X \ar[dddd, "p"] \\
\mc{M}_{1} \times \mc{M}_{1} \times X
    \ar[d, "{(s, t)  \times (s, t)} \times \id_X" left] & & \\
\mc{M}_{0}^4 \times X \ar[d, "\id \times \Delta" left] & & \\
\mc{M}_{0}^4 \times X^4 \ar[d, "\simeq" left] & & \\
(\mc{M}_{0} \times X)^4
    \ar[rr, "pr_i" below] & &
(\mc{M}_{0} \times X)^2
\end{tikzcd}
\end{equation}
where the right vertical map is the vector bundle projection and the bottom horizontal map
is the projection of the product onto the first or last two factors. This yields a canonical
map
\[
\mc{M}_{\Vect(X), 1} \times_{\mc{M}_{\Vect(X)}} \mc{M}_{\Vect(X), 1} \times X
\to[\rho]
pr_1^*(\mc{E}_X^\vee \boxtimes \mc{E}_X)
\times_{B_X^4} pr_2^*(\mc{E}_X^\vee \boxtimes \mc{E}_X)
\]
by the universal property of pullbacks. Now, since
$pr_1^*(\mc{E}_X^\vee \boxtimes \mc{E}_X) \times_{B_X^4}
pr_2^*(\mc{E}_X^\vee \boxtimes \mc{E}_X)$ is the product in $\St_{\s{C}/B_X^4}^1$, it is
also the product in the essential image of $\mc{V}_{B_X^4}$ and thus is the direct sum.
We have a canonical map
\[
pr_1^*(\mc{E}_X^\vee \boxtimes \mc{E}_X) \times_{B_X^4}
pr_2^*(\mc{E}_X^\vee \boxtimes \mc{E}_X) \to[m]
pr_1^*(\mc{E}_X^\vee \boxtimes \mc{E}_X) \otimes
pr_2^*(\mc{E}_X^\vee \boxtimes \mc{E}_X)
\]
by the codomain is just
$\mc{E}_X^\vee \boxtimes \mc{E}_X \boxtimes \mc{E}_X^\vee \boxtimes \mc{E}_X$.

The issue, however, is that there is no immediate trace map
$\mc{E}_X^\vee \boxtimes \mc{E}_X \to \mc{O}_{B_X^4}$. To get around this,
we pullback the bundle
$\mc{E}_X^\vee \boxtimes \mc{E}_X \boxtimes \mc{E}_X^\vee \boxtimes \mc{E}_X$
along the map
$\delta := (\id \times \Delta_{B_X} \times \id) : B_X \times B_X \times B_X \to
B_X \times B_X \times B_X \times B_X$:
\[\begin{tikzcd}
\delta^*pr_1^*\mc{E}_X^\vee
\otimes \delta^*pr_2^*\mc{E}_X
\otimes \delta^*pr_3^*\mc{E}_X^\vee
\otimes \delta^*pr_4^*\mc{E}_X \ar[r] \ar[d] \ar[rd, phantom, "\lrcorner" very near start] &
\mc{E}_X^\vee \boxtimes \mc{E}_X \boxtimes \mc{E}_X^\vee \boxtimes \mc{E}_X \ar[d] \\
B_X^3 \ar[r, "\delta" below] & B_X^4
\end{tikzcd}\]
and observe that $pr_2 \circ \delta = pr_3 \circ \delta$ so that the middle two factors
are of the form $E \otimes E^2$ for a fixed bundle $E$ over $B_X^3$. We further
notice that that left vertical map of \ref{eqn:stpullback_proj} factors through $\delta$
because the maps $s, t$ are given by projections.
This shows that the map $m \circ \rho$ factors through the above pullback. Hence, we have
a map:
\[
\mc{M}_{\Vect(X), 1} \times_{\mc{M}_{\Vect(X)}} \mc{M}_{\Vect(X), 1} \times X
\to[\rho']
\delta^*pr_1^*\mc{E}_X^\vee
\otimes E
\otimes E^\vee
\otimes \delta^*pr_4^*\mc{E}_X
\]
Composing with the trace $\tr : E \otimes E^\vee \to \mc{O}_{B_X^3}$, we get a map
\begin{align*}
\mc{M}_{\Vect(X), 1} \times_{\mc{M}_{\Vect(X)}} \mc{M}_{\Vect(X), 1} \times X
\to[c''] &
\delta^*pr_1^*\mc{E}_X^\vee
\otimes \mc{O}_{B_X^3}
\otimes \delta^*pr_4^*\mc{E}_X \\
\to[\simeq]& \delta^*pr_1^*\mc{E}_X^\vee
\otimes \delta^*pr_4^*\mc{E}_X
\end{align*}
We notice that the codomain of this map is the pullback
\[\begin{tikzcd}
\delta^*pr_1^*\mc{E}_X^\vee \otimes \delta^*pr_4^*\mc{E}_X \ar[r] \ar[d]
    \ar[rd, phantom, "\lrcorner" very near start] &
\mc{E}_{X}^\vee \boxtimes \mc{E}_X \ar[d] \\
B_X^3 \ar[r, "pr_{1, 3}" below] & B_X^2
\end{tikzcd}\]
giving us a map
\begin{align*}
\mc{M}_{\Vect(X), 1} \times_{\mc{M}_{\Vect(X)}} \mc{M}_{\Vect(X), 1} \times X
\to[c''] &
\delta^*pr_1^*\mc{E}_X^\vee
\otimes \mc{O}_{B_X^3}
\otimes \delta^*pr_4^*\mc{E}_X \\
\to[\simeq]& \delta^*pr_1^*\mc{E}_X^\vee
\otimes \delta^*pr_4^*\mc{E}_X \\
\to& \mc{E}_X^\vee \boxtimes \mc{E}_X
\end{align*}
which we will call $c'$.

Further straightforward verification shows that this map composed with the projection
$\mc{E}_X^\vee \boxtimes \mc{E}_X \to B_X \times B_X \to X \times X$
factors through the map:
\[
\mc{M}_{\Vect(X), 1} \times_{\mc{M}_{\Vect(X)}} \mc{M}_{\Vect(X), 1} \times X
\to X \to[X] X \times X
\]
where the first map is the projection onto the $X$ factor. This map corresponds
to a map
\[
c : \mc{M}_{\Vect(X), 1} \times_{\mc{M}_{\Vect(X)}} \mc{M}_{\Vect(X), 1}
    \to \mc{M}_{\Vect(X), 1}
\]
which we take as our desired composition map.
Finally, chasing $U$--points for arbitrary $U$ as for the maps $s, t, e$ before shows that,
a $U$--point $f : U \to \dom(c)$ corresponds to two triples $(E, F, f), (F, G, g)$ and
$c \circ (f, g)$ corresponds to the triple $(E, G, g \circ f)$.


We can now collect the constructions and arguments above into our first
main result:

\begin{thm}\label{thm:arrow_moduli}
There exists a stack $\mc{M}_{\Vect(X), 1} \in \St_\s{C}$ along with maps
\begin{align*}
s, t : \mc{M}_{\Vect(X), 1} \to \mc{M}_{\Vect(X)} \\
e : \mc{M}_{\Vect(X)} \to \mc{M}_{\Vect(X), 1} \\
c : \mc{M}_{\Vect(X), 1} \times_{\mc{M}_{\Vect(X)}} \mc{M}_{\Vect(X), 1}
    \to \mc{M}_{\Vect(X), 1}
\end{align*}
If $\s{C} = \CRing^{\heartsuit, \op}$ and we use the same symbols to denote
the analogous constructions in $\St^1_\s{C}$, we have that
\begin{enumerate}[label=(\roman*), itemsep=0pt]
\item $U$--points of $\mc{M}_{\Vect(X), 1}$ are in bijection with
triples $(E, F, f)$, for vector bundles $E, F$ over $U \times X$
and a map of vector bundles $f : E \to F$.

\item Post-composition with $s$ and $t$ send a $U$--point $f$ of
$\mc{M}_{\Vect(X), 1}$ to $U$--points of $\mc{M}_{\Vect(X)}$
corresponding to the domain and codomain of the map of vector bundles
corresponding to $f$.

\item Post-composition with $e$ sends a $U$-point $f$ of $\mc{M}_{\Vect(X)}$
to the $U$--point of $\mc{M}_{\Vect(X), 1}$ corresponding to the identity
map of the vector bundle classified by $f$.

\item Post-composition with $c$ sends a $U$--point $f$ of $\dom(c)$
given by two triples $(E, F, f), (F, G, g)$ to the triple $(E, G, g \circ f)$.

\item $(\mc{M}_{\Vect(X)}, \mc{M}_{\Vect(X), 1}, s, t, e, c)$ forms an
internal category in $\St^1_{\s{C}, \tau}$. 
\end{enumerate}
\end{thm}
\begin{proof}
Only the last point remains to be shown but this follows from the fact
that for each $U$, we have a category object in $\Set$ given by
\[
(\mc{M}_{\Vect(X)}(U), \mc{M}_{\Vect(X), 1}(U), s, t, e, c)
\]
\end{proof}

\section{Moduli Stacks of Quiver Bundles}
\label{sec:ModStQuiver}

The stack $\mc{M}_{\Vect(X), 1}$ defined in the previous section is the moduli
stack of quiver bundles for the quiver
\[
\Delta^1 = \cdot \to \cdot
\]
In this section, we proceed to define moduli stacks of quiver bundles for more general
quivers, using linear algebraic constructions on universal vector bundle as in the
previous section.
Let $\Delta^n = \Hom_{\Delta}(-, [n]) \in \sSet$ be the standard
$n$--dimensional simplex. 
For brevity of notation, we will fix an object $X \in \St_\s{C}$ and denote
$\mc{M}_0 := \mc{M}_{\Vect(X)}, \mc{M}_1 := \mc{M}_{\Vect(X), 1}$.
We will consider $n + 1$ copies of $\mc{M}_0$ denoted $\mc{M}_{0, j}$ for
$j = 0, 1 \dots, n$, and $\binom{n}{2}$ copies of $\mc{M}_1$, denoted
$\mc{M}_{1, (i, j)}$ for $i < j \in [n] = \set{0, 1, \dots, n}$.
The moduli stack of non-commutative diagrams
of vector bundles over $X$, indexed by the $1$--skeleton $P^n$ of $\Delta^n$,
should simply be the limit, denoted $\mc{M}_{\Vect(X), P^n}$ or just $\mc{M}_{P^n}$,
of the diagram in $\St_\s{C}$ formed by the vertices
\[
\set{\mc{M}_{1, (i, j)}}_{i, j \in [n]} \cup \set{\mc{M}_{0, i}}_{i \in [n]}
\]
and arrows:
\[
\set{\mc{M}_{1, (i, j)} \to[s] \mc{M}_{0, i}}_{i, j \in [n]}
\cup \set{\mc{M}_{1, (i, j)} \to \mc{M}_{0, j}}_{i, j \in [n]}
\]
For example, for $n = 2$, this limit diagram is as follows:
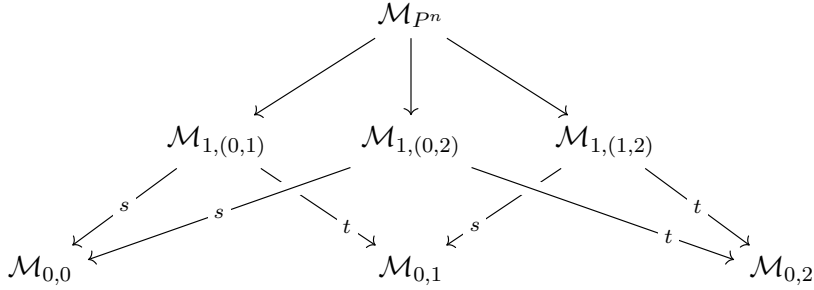
\[\begin{tikzcd}[row sep=large]
& & \mc{M}_{P^n} \ar[ld] \ar[d] \ar[rd] & & \\
&
\mc{M}_{1, (0, 1)} \ar[ld, "s" description] \ar[rd, "t" description, near end] &
\mc{M}_{1, (0, 2)} &
\mc{M}_{1, (1, 2)} \ar[ld, "s" description, near end] \ar[rd, "t" description] & \\
\mc{M}_{0, 0} \ar[from=urr, shift left, crossing over, "s" description] & & \mc{M}_{0, 1} & &
\mc{M}_{0, 2} \ar[from=ull, shift right, crossing over, "t" description, near end]
\end{tikzcd}\]

Now, consider any subset $S = \set{k_0 < k_1 < \cdots < k_m} \subset [n]$.
$\mc{M}_{P^n}$ admits a canonical projection to the pullback:
\[
\mc{M}_{1, (k_0, k_1)} \times_{\mc{M}_{0, k_1}}
\mc{M}_{1, (k_1, k_2)} \times_{\mc{M}_{0, k_2}} \cdots \times_{\mc{M}_{0, k_{m - 1}}}
\mc{M}_{1, (k_{m - 1}, k_m)}
\]
which, in turn, admits a composition morphism to $\mc{M}_{1, (k_0, k_m)}$ by iterating
composition morphism of \cref{thm:arrow_moduli}.
\begin{warn}
This iterated composition morphism is defined up to contractible choice in the
$\infty$--categorical case.
\end{warn}
This gives us maps of the following form for each such $S \subset [n]$:
\[
c_S : \mc{M}_{P^n} \to \mc{M}_{1, (k_0, k_m)}
\]
We want all composites with the same source and target vertices to be equivalent and
this yields the following definition.

\begin{defn}[Moduli Stack of Standard Quiver Bundles]
We define $\mc{M}_{\Vect(X), \Delta^n}$ to be the pullback:
\[\begin{tikzcd}
\mc{M}_{\Vect(X), \Delta^n} \ar[r] \ar[d] \ar[rd, phantom, "\lrcorner" very near start] &
\mc{M}_{\Vect(X), P^n}
    \ar[d, "\set{c_S}_{S \subset [n]}"] \\
\prod_{i < j \in [n]} \mc{M}_{1, (i, j)} \ar[r, "\Delta" below] &
\prod_{i < j \in [n]}
\mathop{\mathop{\prod_{S \subset [n],}}_{\min S = i,}}_{\max S = j}
    \mc{M}_{1, (i, j)}
\end{tikzcd}\]
where the bottom horizontal map is given by diagonal morphisms.
We will call $\mc{M}_{\Vect(X), \Delta^n}$ the moduli stack of standard $n$--quivers.
\end{defn}

Let $\iota : [m] \to[] [n]$ be an injection and consider the ordered set
\[
J(i) = \set{j_0 = \iota(i) < j_1 = \iota(i) + 1 < \dots < j_k = \iota(i + 1)}
\]
Consider the map
$c_{J(i)} : \mc{M}_{P^n} \to \mc{M}_{1, (j_0, j_k)} \simeq \mc{M}_{1, (i, i + 1)}$
as defined before.
The collection of these maps for all $i \in [m]$ induce a map
$\mc{M}_{\Vect(X), \iota} : \mc{M}_{P^n} \to \mc{M}_{P^m}$, written $\mc{M}_\iota$ for
brevity.
Now, let $\sigma : \Delta^m \to \Delta^n$ be a map corresponding to a surjection
$\sigma : [q] \to[] [n]$. We then have diagonal maps
\[
\mc{M}_{0, i} \to[\Delta] \prod_{i' \in \sigma^{-1}(i)} \mc{M}_{0, i'}
\]
and maps
\begin{align*}
& \mc{M}_{1, (i, j)} \\
\to[(s, \id)] & \mc{M}_{0, s(i)} \times \mc{M}_{1, (i, j)} \\
\to[(e, e, \dots, e) \times \id] &
    \prod_{i' \in \sigma^{-1}(i) \setminus \set{\max \sigma^{-1}(i)}}
        \mc{M}_{1, (i', i' + 1)} \times \mc{M}_{1, (\max \sigma^{-1}(i) - 1, \max \sigma^{-1}(i))}
\end{align*}
These induce a map $\mc{M}_{\sigma} : \mc{M}_{P^n} \to \mc{M}_{P^q}$, again
written $\mc{M}_{\Vect(X), \sigma}$ for brevity.
The motivation for the map $\mc{M}_\iota$ is that it sends a chain of
$n$ composeable arrows of vector bundles to a chain of $m$ composeable arrows by composing
subchains of arrows prescribed by the function $\iota$.
The motivation behind $\mc{M}_\sigma$ is
that it sends a chain of $n$ composeable arrows to a chain of $q$ composeable arrows by
inserting identity morphisms in a pattern encoded by $\sigma$. In fact, it can be shown
that the maps $\mc{M}_\iota$ and $\mc{M}_\sigma$ induce corresponding canonical maps
$\mc{M}_{\Delta^n} \to \mc{M}_{\Delta^m}, \mc{M}_{\Delta^n} \to \mc{M}_{\Delta^q}$, by using the universal property of pullbacks ---
the idea is: composing and inserting identities respects commutativity of diagrams.
Finally, it can also be shown that the maps
$\mc{M}_{\iota} \circ \mc{M}_{\sigma} = \mc{M}_{\sigma \circ \iota}$ whenever
$\iota \circ \sigma$ exists and since the morphisms of the simplicial indexing category
$\Delta$ can be factorized into injections and surjections,
we have the following results:

\begin{cor}
The assignment $[n] \mapsto \mc{M}_{\Vect(X), \Delta^n}$ as defined above gives a
simplicial object $\mc{M}_{\Vect(X), -} : \Delta^\op \to \St_{\sC}$.
\end{cor}

\begin{defn}[Moduli Stack of Quiver Bundles]\label{defn:mod-st-quiv-bun}
We define a functor $\mc{M}_{\Vect(X), -} : \sSet^\op \to \St_\sC$ by right Kan extension
of the functor of the previous corollary along the Yoneda embedding
$\Delta^\op \to \sSet^\op$. Given a simplicial set $I \in \sSet$, we define
the stack $\mc{M}_{\Vect(X), I}$ to be the moduli stack of quiver bundles of shape $I$.
\end{defn}

\begin{cor}
In the context of the previous definition and in the case that
$\s{C} = \CRing^{\heartsuit, \op}$ and analogous constructions in $\St_\sC^1$,
$U$--points of $\mc{M}_{\Vect(X), I}$ correspond to diagrams of vector bundles over
$U \times X$ indexed by $I$. If $I$ is a quasicategory, then these are commutative diagrams.
\end{cor}

We now discuss a sufficient condition for moduli stacks of quiver
bundles of finite shape (that is, indexed by finite simplicial sets) to be algebraic,
by which we mean Artin.
We will only address the question of algebraicity for a class of base
$1$--stacks relevant for some applications we have in mind,
and leave a comprehensive treatment for future work.

\begin{lem}\label{lem:alg-quiv-bun-mod}
Let $\AlgSt_{\s{C}, \tau} \subset \St_{\s{C}, \tau}$ be a full subcategory of $\s{C}$
closed under pullbacks and containing the terminal object of $\s{C}$. For any
$X \in \s{C}$, if
$\mc{M}_{\Vect(X)}, \Map(X, X \times X), \Map(X, \mc{E}_X^\vee \boxtimes \mc{E}_X)
\in \AlgSt_{\s{C}}$,
then for any finite simplicial set $I$, $\mc{M}_{\Vect(X), I} \in \AlgSt_{\s{C}}$.
\end{lem}
\begin{proof}
For a finite simplicial set $I$, $\mc{M}_{\Vect(X), I}$ is defined as a finite limit
involving $\mc{M}_{\Vect(X), 1}$ and $\mc{M}_{\Vect(X)}$. On the other hand,
$\mc{M}_{\Vect(X), 1}$ is defined as a pullback involving the terminal object in
$\St_{\s{C}}$, $\Map(X, X \times X)$ and $\Map(X, \mc{E}_X^\vee \boxtimes \mc{E}_X)$. Hence,
if the two latter objects are in $\AlgSt_\s{C}$, then so is $\mc{M}_{\Vect(X), 1}$.
The result now follows from the fact that finite limits can be computed by using
pullbacks and terminal objects. 
\end{proof}

\begin{defn}[Stabilizers of Points {\cite[Definition 5.19]{AK22}}]
Let $V \in \St_\s{C}$ and $x : A \to V$ be a point of $V$ for $A \in \s{C}$.
The stabilizer of $x$ is defined to be the pullback:
\[\begin{tikzcd}
\Stab_V(x) \ar[r] \ar[d] \ar[rd, phantom, "\lrcorner" very near start] & A \ar[d, "x"] \\
A \ar[r, "x" below] & V
\end{tikzcd}\]
We say $V$ has affine stabilizers if for all $x : A \to V$ and for all $A \in \s{C}$,
$\Stab_V(x) \in \s{C}$.
\end{defn}

\begin{lem}\label{lem:affine-stabilizers}
For any diagram of stacks $Y \to W \ot Z$, suppose $Y, Z$ have affine stabilizers. Then so does $Y \times_W Z$.
\end{lem}
\begin{proof}
For  $A \in \CRing^{\heartsuit, \op}_{R/}$, and any point $x : A \to Y \times_W Z$,
we have the two points $x_V : A \to[x] Y \times_W Z \to V, V = Y, Z$.
We will show that
\[
\Stab_{Y \times_W Z}(x) \simeq \Stab_{Y}(x_Y) \times_A \Stab_Z(x_Z)
\]
which is in $\s{C}$ because the Yoneda embedding $\s{C} \hto \St_\s{C}$ preserves limits.
Let $f_Y : P \to \Stab_Y(x_Y), f_Z : P \to \Stab_Z(x_Z)$ such that the following square
commutes:
\[\begin{tikzcd}
P \ar[r, "f_Y"] \ar[d, "f_Z" left] & \Stab_Y(x_Y) \ar[d] \\
\Stab_Z(x_Z) \ar[r] & A
\end{tikzcd}\]
This diagram can be extended to a commutative diagram:
\[\begin{tikzcd}
P \ar[r, "f_Y"] \ar[d, "f_Z" left] & \Stab_Y(x_Y) \ar[d] \ar[rdd, bend left] & \\
\Stab_Z(x_Z) \ar[r] \ar[rrd, bend right] & A \ar[rd, "x"] & \\ & &
Y \times_W Z
\end{tikzcd}\]
This, in turn, gives a unique map $P \to \Stab_{Y \times_W Z}(x)$ by the universal property
of $\Stab_{Y \times_W Z}(x) = \lim\br{A \to Y \times_W Z \ot A}$, yielding the result.
\end{proof}

\begin{thm}\label{thm:algebraicity}
Let $R$ be a ring and
$(\s{C}, \tau) = (\CRing^{\heartsuit, \op}_{R/}, \mathrm{fppf})$ and let
$\AlgSt_{\s{C}, \tau}$ be the full subcategory of Artin stacks.
Let $X \to \Spec(R) \in \St_{\s{C}}^1$ satisfy the following:
\begin{enumerate}[label=(\roman*), itemsep=0pt]
\item $X$ is in $\AlgSt_{\s{C}} = \AlgSt_{\s{C}, \tau}$.
\item $X$ is proper, flat, of finite presentation, and has affine stabilizers.
\item $\mc{M}_{\Vect(X)}$ is in $\AlgSt_\s{C}$.
\item $\mc{M}_{\Vect(X)}$ is quasi-separated.
\item $\mc{M}_{\Vect(X)}$ is locally of finite presentation.
\end{enumerate}
Then, for any finite simplicial set $I$, the following hold:
\begin{enumerate}[label=(\roman*), itemsep=0pt]
\item $\mc{M}_{\Vect(X), I} \in \AlgSt_{\s{C}}$.
\item $\mc{M}_{\Vect(X), I}$ is locally of finite presentation and quasi-separated.
\item If $X \times X$ and $\mc{E}_X^\vee \boxtimes \mc{E}_X$ have affine diagonal maps, and
$\mc{M}_{\Vect(X)}$ has affine stabilizers,
then $\mc{M}_{\Vect(X), I}$ has affine stabilisers.
\end{enumerate}
\end{thm}
\begin{proof}
We first notice that $X \times X$ is of finite presentation and hence, in
particular, locally of finite presentation and quasi-separated. In addition,
it has affine stabilizers by \cref{lem:affine-stabilizers}, as $X$ does by
hypothesis.
Denoting $\mc{E} := \mc{E}_X^\vee \boxtimes \mc{E}_X$ and
$B := \mc{M}_{\Vect(X)} \times X \times \mc{M}_{\Vect(X)} \times X$,
$\mc{E}$ being affine over the quasi-separated algebraic stack $B$
implies the following:
\begin{enumerate}[label=(\roman*), itemsep=0pt]
\item $\mc{E}$ is algebraic as the morphism $p : \mc{E} \to B$ is affine, and hence
representable by algebraic spaces, with algebraic codomain.\footnote{To see this, first observe that
we can pull back the atlas of $B$ to $\mc{E}$ which will be an atlas since $\mc{E} \to B$ is
representable by algebraic spaces.
Next, consider the composite
$\mc{E} \to[\Delta_{\mc{E}}] \mc{E} \times \mc{E} \to[\id \times p] \mc{E} \times B$ which is
$\mc{E} \to[{(\id, p)}] \mc{E} \times B$. Applying \cite[Lemma 6.8(3)]{CW17} twice, we can deduce
that the diagonal of $\id \times p$ is representable, and we can verify by a diagram chase that
$(\id, p)$ is representable by algebraic spaces.
This implies $\mc{E} \to[\Delta] \mc{E} \times \mc{E}$ is schematic by
\cite[Lemma 6.8(4)]{CW17}.}
\item $\mc{E} \to B$ is quasi-separated as an affine morphism is
quasi-separated by
\cite[\href{https://stacks.math.columbia.edu/tag/01S7}{Tag 01S7}]{stacks-project}.
Therefore, it is quasi-separated over $\Spec(R)$ as the composite
$\mc{E} \to B \to \Spec(R)$ of quasi-separated morphisms quasi-separated.
\end{enumerate}
Next, $\mc{E}$ being the relative spectrum of the symmetric algebra of a finite
locally free sheaf over $B$ implies the following:
\begin{enumerate}[label=(\roman*), itemsep=0pt]
\item $\mc{E}$ has trivial stabilizers as its functor of points is valued in $\Set$.
\item $\mc{E}$ is locally of finite presentation over $B$ and hence over
$\Spec(R)$, since $B$ itself is locally of finite presentation over $\Spec(R)$ as it
is a product of stacks locally of finite presentation.
\end{enumerate}
These observations allow us to apply
\cite[Theorem 1.2(i)]{HR19} to deduce that the mapping stacks
$\Map(X, X \times X)$ and $\Map(X, \mc{E}_X^\vee \boxtimes \mc{E}_X)$ are in
$\AlgSt_{\s{C}}$. Then, \cref{lem:alg-quiv-bun-mod} implies (i).

For (ii), we observe that the properties involved are preserved by
pullback and hence finite limits of stacks satisfying these two properties have them. So,
it suffices to verify that the mapping stacks before have these properties which follows
by applying \cite[Theorem 1.2(ii)]{HR19}.

For (iii), we first prove that
$\Map(X, X \times X), \Map(X, \mc{E}_X^\vee \boxtimes \mc{E}_X)$ have affine stabilizers.
To do so, it suffices to show that $X \times X$ and $\mc{E}_X^\vee \boxtimes \mc{E}_X$
have affine diagonals by \cite[Theorem 1.2(iii)]{HR19}, but this is hypothesised in
(iii) above. Furthermore, the stack $\mathrm{pt}$ is an affine scheme
and hence has trivial stabilizers. By \cref{lem:affine-stabilizers}, it follows that
$\mc{M}_{\Vect(X), 1}$ has affine stabilizers.
Again, since $\mc{M}_{\Vect(X), I}$ is a finite limit computed using pullbacks involving
the terminal object, $\mc{M}_{\Vect(X)}$ and $\mc{M}_{\Vect(X), 1}$, which all have
affine stabilizers, \cref{lem:affine-stabilizers} implies that $\mc{M}_{\Vect(X), I}$
has affine stabilizers.
\end{proof}

\begin{cor}
If $X$ is a projective variety over a field, then $\mc{M}_{\Vect(X), I}$ is algebraic
for all finite simplicial sets $I$.
\end{cor}
\begin{proof}
It suffices to observe that $\mc{M}_{\Vect(X)}$ is algebraic in this case by
\cite[Theorem 10.20]{AK22} and that $\mc{M}_{\Vect(X)}$ is locally of finite
presentation by \cite[Theorem 7.10]{CW17}.
\end{proof}

\begin{rmk}
To the best of our knowledge, the hypotheses of \cref{thm:algebraicity} are
not the most relaxed possible.
There are more general results available on algebraicity of mapping
stacks in both the classical \cite[Theorem 6.22]{AHR23} and the
derived settings \cite[Theorem 5.1.1]{HP19}. Furthermore, even more
relaxed hypotheses might be sufficient if we rephrase the stack
$\mc{M}_{\Vect(X), 1}$ as a certain Weil restrition as follows (this was
suggested to us by Daniel Halpern-Leistner at the
$2^{\mathrm{nd}}$ Simons Math Summer Workshop on ``Moduli'' in July 2024):
let $\mc{E}'$ denote the pullback of $\mc{E}_X^\vee \boxtimes \mc{E}_X$ to
$\mc{M}_{\Vect(X)} \times \mc{M}_{\Vect(X)} \times X$ along the map
$\id_{\mc{M}_{\Vect(X)}} \times \id_{\mc{M}_{\Vect(X)}} \times \Delta_X$,
then the Weil restriction of $\mc{E}'$ to
$\mc{M}_{\Vect(X)} \times \mc{M}_{\Vect(X)}$ via the product
projection has the same functor of points as $\mc{M}_{\Vect(X), 1}$, which
then allows us to examine the possibility of using algebraicity results
on Weil restrictions such as \cite[Theorem 5.1.14]{HP19}.
In future work, we will attempt to use such results to prove
algebraicity of moduli stacks of quiver bundles over more general
stacks.
We will not address any results about smoothness in this paper but
we plan to do so as well in future work.
\end{rmk}

\begin{rmk}[Meaning of categorification]
\label{rmk:meaning}
In the previous section, we constructed moduli stacks parametrizing the morphisms
involved in the moduli problem for vector bundles.
We have discussed how these morphism moduli stacks along with the original moduli
stacks form internal categories.
In this section, we have
used that stack to construct an algebraic moduli stack parametrizing vector bundle
diagrams of any fixed finite shape.
In fact, given a simplicial set $I$, we can consider the
moduli stack $\mc{M}_{\Vect(X), I \times \Delta^1}$ to be the moduli stack of morphisms
of quiver bundles of shape $I$ and it is not much different to show that this forms
an internal category. It is in this sense, that we have categorified the
moduli problem of vector bundles as well as quiver bundles. That is,
we have produced a notion of a category stack parametrizing these objects and their
morphisms as opposed to just a single stack parametrizing the objects.
\end{rmk}

\section{Quiver Varieties}
\label{sec:QuiverVarieties}

It is imperative from our title that we discuss how to recover
quiver varieties and Nakajima quiver varieities from our framework and we now provide
a sketch to accomplish this.
We emphasize that some of the statements made in this section are
conjectural and we hope we have made it clear when it is. We intend to complete this
work in a future paper.

\subsection{Quiver Stacks and Quiver Varieties}

Pick any quiver $I$ in the usual sense of quiver theory ---
this is a simplicial set with no $2$--simplices. Assuming we are working over a base
ring $k$, take for the base space $X$, the point $\Spec(k)$. Then,
vector bundles over $X$ are the affine schemes over $k$ and
$\mc{M}_{\Vect(X), I}$ is the representation stack $\mathrm{Rep}(I)$.
The difference is that for constructing the classical representation space,
one needs to choose
dimension vectors in order to write down the space as a direct sum of vector space
Hom's. The stack $\mc{M}_{\Vect(X), I}$ parametrizes all finite dimensional
representations irrespective of dimension vectors.
One might then try to translate
the group action used in quiver theory to construct quiver varieities to a group action
on this stack and then take a GIT quotient as the quiver variety.
However, we will phrase the construction using simplicial language.
That is, we will construct a different simplicial set $\tilde{I}$ whose
representations in $\Vect(\Spec(k))$ correspond to triples $(g, \rho, \psi)$, where
$g$ is an element of the automorphism group of representations of $I$, and $\rho$ and
$\psi$ are representations of $I$, such that $g \cdot \rho = \psi$. Roughly speaking,
$\tilde{I}$ will contain two copies of $I$ joined at the corresponding vertices with
a pair of oppositely directed edges and a pair of $2$--simplices encoding that they are
isomorphisms. For each edge $e \in I$, denoting its second copy in $\tilde{I}$ as $e'$,
and the two edges $g_{0, e} : d_0e \to d_0e'$ and $g_{1, e'} : d_1e \to d_1e'$ in
$\tilde{I}$, we will have two $2$--simplices encoding the relation
$g_{1, e}e = e'g_{0, e}$.
If the quiver $I$ has framing nodes, we have to be slightly more careful.

We now give a more explicit construction. First, framing nodes can be encoded as a
function $f : I_0 \to \set{0, 1}$, where $f^{-1}(0)$ are the usual nodes and
$f^{-1}(1)$ are the framing nodes.  The vertex set
of $\tilde{I}$ is then defined to be:
\[
\tilde{I}_0 := f^{-1}(0) \amalg f^{-1}(0) \amalg f^{-1}(1)
\]
which has two copies of each regular node of $I$ and one copy of each framing node.
To describe the edge
set, first consider the set $S := (fd_0, fd_1)^{-1}(1, 1)$ of edges whose source
and target are both framing nodes, and the set $I_1 \setminus S$ of edges which has
a source or a target that is not a framing node. For each regular node $v \in f^{-1}(0)$,
we consider symbols $g_v, g_v^{-1}$, and define
\[ T := \set[g_v]{v \in f^{-1}(0)} \amalg \set[g_v^{-1}]{v \in f^{-1}(0)} \]
These are to
be thought of as isomorphism edges going from one copy of $v$ to another in $\tilde{I}$.
For $e \in I_1 \setminus S$, we also consider tuples $(g_{d_1(e)}, e)$ and think of
them as the composite $g_{d_1(e)} \circ e$. We write:
\[
T' := \set[{(g_v, e)}]{e \in I_1 \setminus S, v = d_1(e)}
\]
Then, we define the edge set to be:
\[
\tilde{I}_1 := (I_1 \setminus S) \amalg (I_1 \setminus S) \amalg S \amalg T \amalg T'
\]
where the two copies of $I_1 \setminus S$ consist of two copies of each edge in $I$
incident to a regular node, $S$ consists of one copy of each edge incident to
framing nodes, and $T$ consists of the isomorphism edges connecting the two copies
of each regular node, $T'$ consists of composites of edges between regular nodes
and isomorphism edges. We leave it to the reader to define the face and degeneracy maps
to make these ideas precise.

Next, we describe the $2$--simplices of $\tilde{I}$ which impose the necessary relations.
For each $v \in f^{-1}(0)$, denoting its second copy in $\tilde{I}$ as $v'$,
we consider $2$--simplices of the form:
\[
\alpha_v := \begin{tikzcd}
& v \ar[ld, equal] \ar[rd, "g_v"] & \\
v & & v' \ar[ll, "g_v^{-1}"]
\end{tikzcd} \hspace{5em}
\beta_v := \begin{tikzcd}
& v' \ar[ld, "g_v^{-1}" above left] \ar[rd, equal] & \\
v \ar[rr, "g_v" below] & & v'
\end{tikzcd}\]
For $e \in (fd_0, fd_1)^{-1}(0, 0) \subset I_1 \setminus S$ --- that is, an edge
incident to no framing nodes ---
we denote its second copy in $\tilde{I}_1$ as $e'$, and
we consider $2$--simplices that glue to squares of the form:
\[ \gamma_e := \begin{tikzcd}[column sep=huge, row sep=huge]
d_0(e) \ar[r, "e"] \ar[d, "g_{d_0(e)}" left] \ar[rd, "{(g_{d_1(e)}, e)}" description] &
d_1(e) \ar[d, "g_{d_1(e)}"] \\
d_0(e') \ar[r, "e'" below] & d_1(e')
\end{tikzcd}\]
We call the upper right $2$--simplex $\alpha_e$ and the lower left $2$--simplex
$\beta_e$.

For $e \in (fd_0, fd_1)^{-1}(0, 1) \subset I_1 \setminus S$ --- that is,
an edge with a framing node as a target --- we again denote its second copy in $\tilde{I}$
as $e'$, and consider a $2$--simplex of the form:
\[ \delta_e := \begin{tikzcd}
d_0(e) \ar[rd, "e" above right] \ar[dd, "g_{d_{0}(e)}" left] & \\
& d_1(e) = d_1(e') \\
d_0(e') \ar[ru, "e'" below right] &
\end{tikzcd}\]
For $e \in (fd_0, fd_1)^{-1}(1, 0)$, we consider a similar simplex:
\[ \epsilon_e := \begin{tikzcd}
& d_1(e) \ar[dd, "g_{d_{1}(e)}"]  \\
d_0(e) = d_0(e') \ar[ru, "e" above left] \ar[rd, "e'" below left] &  \\
& d_1(e')
\end{tikzcd}
\]
We then define:
\begin{align*}
\tilde{I}_2 :=& \set[\alpha_v]{v \in f^{-1}(0)} \amalg \set[\beta_v]{v \in f^{-1}(0)} \\
\amalg& \set[\gamma_e]{e \in (fd_0, fd_1)^{-1}(0, 0)} \\
\amalg& \set[\delta_e]{e \in (fd_0, fd_1)^{-1}(0, 1)}
\amalg \set[\epsilon_e]{e \in (fd_0, fd_1)^{-1}(1, 0)}
\end{align*}
The face and degeneracy maps are the obvious ones. Finally, we define the higher
$\tilde{I}_n, n > 2$ to only consist of degenerate simplices. A point to note here
is that this construction of $\tilde{I}$ is only simplicial, and highly non-categorical.
In particular, $\tilde{I}$ does not have simplices to witness composites of all
composeable chains of arrows, it only has enough $2$--simplices to encode the equivalence
relation given by the group action on the representation space of $I$. However, we note
that quiver theory requires this non-categorical aspect: the quiver $I$ we begin with
is not a category and its representations are not functors, but maps of simplicial sets
$I \to \Vect(\Spec(k))$; furthermore, the equivalence relation given
by the group action on representations selectively takes into account composites of
arrows.

\begin{exm}
Let $I$ be the graph:
\[\begin{tikzcd}
a \ar[r, "e_{ab}"] & b \ar[r, "e_{bc}"] & c
\end{tikzcd}\]
where $a, b$ are regular nodes --- $f(a) = 0 = f(b)$ --- and $c$ is a framing node ---
$f(c) = 1$. Then, $\tilde{I}$ is a simplicial set of the form (not a category!):
\[\begin{tikzcd}[row sep=small, column sep=huge]
a \ar[r, "e_{ab}"] \ar[dd, shift right, "g_a" left] &
b \ar[dd, shift right, "g_b" left] \ar[rd, "e_{bc}" above right] \\
& & c \\
a' \ar[r, "e_{ab}'" below] \ar[uu, shift right, "g_a^{-1}" right] &
b' \ar[uu, shift right, "g_b^{-1}" right] \ar[ru, "e_{bc}'" below right] &
\end{tikzcd}\]
where we have $2$--simplices witnessing the identities
$g_a^{-1}g_a = \id_a$, $g_b^{-1}g_b = \id_b$, $g_ag_a^{-1} = \id_{a'}$,
$g_bg_b^{-1} = \id_{b'}$, $g_be_{ab} = e'_{ab}g_a$,
$e_{bc}'g_b = e_{bc}$, but no $2$--simplex witnessing a composite
$e_{bc}e_{ab}$, for example.
\end{exm}

\begin{rmk}
 It is likely that there is a more ``invariant'' way to construct the graph $\tilde{I}$
or another one categorically equivalent to it (note that categorical equivalence makes
sense for simplicial sets and not just quasicategories) as a colimit involving
$I$ and its simplicial subsets, so that its representations encode the group
action in a similar way. However, we will not pursue this idea for now.
\end{rmk}

With this setup, we observe that we have two obvious inclusions
$\iota, \kappa : I \to \tilde{I}$ and we can define the moduli stack of quiver
representations to be the $(2, 1)$--coequalizer
\[
\mathrm{Quiver}(I) := \colim\left(\begin{tikzcd}[column sep=huge]
\mc{M}_{\Vect(X), \tilde{I}}
    \ar[r, shift left, "\mc{M}_{\Vect(X), \iota}"]
    \ar[r, shift right, "\mc{M}_{\Vect(X), \kappa}" below] &
\mc{M}_{\Vect(X), I}
\end{tikzcd}\right)
\]
Then, one can hope to apply ``Beyond GIT'' in the sense of
\cites{HL22, AHLH24} to recover quiver varieties in the usual sense, but these
will be quiver varieties with respect to all choices of dimension vectors, and hence
a disjoint union of quiver varieities for each choice of dimension vector.

\subsection{Nakajima Quiver Varieties}

Next, we would like to recover Nakajima quiver varieties, as originally introduced in \cite{MR1302318}, out of our framework. The first
step towards this is to add framing nodes for each vertex. We accomplish this by first
gluing a copy of the graph:
\[\begin{tikzcd}
0 \ar[r] & 1
\end{tikzcd}\]
to $I$ at each vertex $v$, where $v$ and $0$ are identified. This can be realized as
the strict pushout of simplicial sets:
\[\begin{tikzcd}
\mathrm{sk}_0I \ar[r] \ar[d] & I \ar[d] \\
\coprod_{v \in I_0} (0 \to 1) \ar[r] &
I^{\mathrm{fr}} \ar[lu, phantom, "\ulcorner" very near start]
\end{tikzcd}\]
where $\mathrm{sk}_0I$ is the $0$--skeleton of $I$ or the discrete simplicial set
consisting of the vertices of $I$, and the left vertical map sends a vertex
$v$ to the source vertex of the copy of $0 \to 1$ corresponding to $v$.
The next step is to double the edges, as they
say in quiver theory.
Given a quiver $J$, we take its double to be the strict pushout:
\[\begin{tikzcd}
\mathrm{sk}_{0}J \simeq \mathrm{sk}_{0}J^\op \ar[r] \ar[d] &
J \ar[d] \\
J^\op \ar[r] & d(J) \ar[lu, phantom, "\ulcorner" very near start]
\end{tikzcd}\]
Indeed, this takes a copy of $J$ and a copy of $J^\op$ and glues them together
at the vertex set. We are then interested in the representations of
$d(I^\mathrm{fr})$ but those in the zero locus of the natural moment maps or so-called preprojective conditions\footnote{If $x\in\mbox{Hom}(\mathbb C^s,\mathbb C^t)$ is a representation of an arrow in $J$ and if $y\in\mbox{Hom}(\mathbb C^t,\mathbb C^s)$ is a representation of its reversed-orientation counterpart in $J^\op$, then there is an identification of $y$ with a unique element in $T^*_x\mbox{Hom}(\mathbb C^s,\mathbb C^t)$, coming from the trace pairing, where $T^*_x$ refers to the (algebraic) cotangent space at $x$. Working over $\bC$, these pairings allow us to find within the Nakajima quiver variety an open dense subset that may be identified with the (geometric) cotangent bundle of the ordinary quiver variety (associated to $J$), and the moment map conditions equip the Nakajima quiver variety with quaternionically-commuting complex structures and corresponding symplectic forms, all of which are compatible with a metric inherited by the quotient (cf. \cite{HKLR87,MR1302318}).  This package of differential-geometric data is the so-called hyperk\"ahler structure.  One of these complex structures extends the one on the cotangent bundle uniquely. This feature is essential in many of the applications of Nakajima quiver varieties --- for example, in integrable systems \cite{RS20}.  Finally, it is worth remarking that satisfying the moment map conditions is the GIT stability condition for this moduli problem.}.
Moment maps are written using composition, direct sum,
addition, and scaling (by constants) of linear maps indexed by
arrows in the quiver whose
representation we are considering. While it is easy to handle composition in our language,
it is not clear how one should formalize direct sums, addition, subtraction, scaling.
As a start, we can treat direct sums in much the same way as the map $\otimes_1$ in
\cref{guess:symm-mon-moduli}, giving a map:
\[
\oplus_1 : \mc{M}_{\Vect(X), 1} \times \mc{M}_{\Vect{X}, 1} \to \mc{M}_{\Vect(X), 1}
\]
For the rest, we can
first observe that, even when $X$ is a stack over $k$, for each $\lambda \in k$,
there are maps:
\[
s_\lambda : \mc{E}_X^\vee \boxtimes \mc{E}_X \to \mc{E}_X^\vee \boxtimes \mc{E}_X
\]
that, viewing the domain and codomain as the respective sheaves of $k$--modules,
scales a section by $\lambda$, and
\[
+ : \mc{E}_X^\vee \boxtimes \mc{E}_X \oplus \mc{E}_X^\vee \boxtimes \mc{E}_X \to
\mc{E}_X^\vee \boxtimes \mc{E}_X
\]
that sends two sections to their sum. Some formal arguments are then necessary to show
that these yield maps of moduli stacks:
\[
s_\lambda : \mc{M}_{\Vect(X), 1} \to \mc{M}_{\Vect(X), 1}
\]
that scales a morphism of vector bundles over $X$ by $\lambda$,
and for each graph of the forms:
\[J = \begin{tikzcd}
\cdot \ar[r, "a", bend left] \ar[r, "b" below, bend right] & \cdot
\end{tikzcd} \text{ or } J = \begin{tikzcd}
\cdot \arrow[loop left]{l}{a} \arrow[loop right]{r}{b}
\end{tikzcd}\]
a map:
\[
+ : \mc{M}_{\Vect(X), J} \to \mc{M}_{\Vect(X), 1}
\]
that sends the morphisms indexed by $a$ and $b$ to their sum. One can then see that
that this is enough to write down moment maps using only the moduli stacks we have
constructed, by taking the various projections
$\mc{M}_{\Vect(X), d(I^\mathrm{fr})} \to \mc{M}_{\Vect(X), K}$ for various choices
of simplicial subsets $K \subset I$ along with the maps $\oplus_1, s_\lambda, +$
discussed above and the composition map $c$ of \cref{thm:arrow_moduli}.
One then argues that moment maps will be of the form
\[
\mu : \mc{M}_{\Vect(X), d(I^\mathrm{fr})} \to \prod_{v \in I_0} \mc{M}_{\Vect(X), 1}
\]
We can compare this with moment maps in usual quiver theory ---
they take values in the endomorphism spaces of the vector spaces in a
representation. In our case, we for each vertex $v \in I_0$, a copy of
$\mc{M}_{\Vect(X), 1}$ which, in particular, also has the endomorphisms.

To take a zero locus, one first needs a map
\[
0 : \mathrm{pt} \to \prod_{v \in I_0} \mc{M}_{\Vect(X), 1}
\]
that picks out the zero maps of vector bundles. We claim that the zero
section of $\mc{E}_X^\vee \boxtimes \mc{E}_X$ can be used to construct such a map. The
zero locus of the moment maps is then simply the pullback of stacks:
\[\begin{tikzcd}
Z_\mu \ar[r] \ar[d] \ar[rd, phantom, "\lrcorner" very near start] &
\mc{M}_{\Vect(X), d(I^\mathrm{fr})} \ar[d, "\mu"] \\
\mathrm{pt} \ar[r, "0" below] &
\prod_{v \in I_0} \mc{M}_{\Vect(X), 1}
\end{tikzcd}\]
We then consider the following diagram:
\[\begin{tikzcd}[column sep=huge]
\mc{M}_{\Vect(X), \widetilde{d(I^\mathrm{fr})}}
    \ar[r, shift left, "\mc{M}_{\Vect(X), \iota}" above]
    \ar[r, shift right, "\mc{M}_{\Vect(X), \kappa}" below] &
\mc{M}_{\Vect(X), d(I^\mathrm{fr})} \ar[r, "!"] \ar[rd, "\mu" below left] &
\mathrm{Quiver}(d(I^\mathrm{fr})) \ar[d, dashed, "\mu'"] \\
& & \prod_{v \in I_0} \mc{M}_{\Vect(X), 1}
\end{tikzcd}\]
where, if we can show that
$\mu \circ \mc{M}_{\Vect(X), \iota} \simeq \mu \circ \mc{M}_{\Vect(X), \kappa}$, we
obtain a unique-up-to-2-isomorphism map
$\mu'$, shown as a dashed arrow, making the right-most triangle
commute. This is the analogue of the moment map
being invariant with respect to the group action whose quotient gives
the Nakajima quiver variety in the usual theory.
Here, of course, $\iota$ and $\kappa$ are the two inclusions
$d(I^\mathrm{fr}) \hto \widetilde{d(I^\mathrm{fr})}$ and $!$ is the canonical
map to the coequalizer.
Once we have this, since taking pullbacks in functorial, we get the following
commutative diagram:
\[\begin{tikzcd}
Z_\mu \ar[rr] \ar[dd] \ar[rd] & &
\mc{M}_{\Vect(X), d(I^\mathrm{fr})} \ar[dd, "\mu" near start] \ar[rd, "!"] & \\ &
Z_{\mu'} \ar[rr, crossing over] & &
\mathrm{Quiver}(d(I^\mathrm{fr})) \ar[dd, dashed, "\mu'"] \\
pt \ar[rr, "0" below, near end] \ar[rd, equal] & &
\prod_{v \in I_0} \mc{M}_{\Vect(X), 1} \ar[rd, equal] & \\ &
pt \ar[rr, "0" below] \ar[from=uu, crossing over] & & \prod_{v \in I_0} \mc{M}_{\Vect(X), 1}
\end{tikzcd}\]
We can then hope to show that $Z_{\mu'}$ is equivalent
as a stack to the classical Nakajima quiver variety associated to $I$.

\subsection{Quiver Mutations}

The constructions discussed here and
the construction of the moduli stack of Higgs bundles we give in the next section
can be described as
using the internal category structure of the moduli stacks $\mc{M}_{\Vect(X), 1}$
and $\mc{M}_{\Vect(X)}$ to construct a moduli stack of quiver bundles for a different
shape of quiver. That is, the integrability condition for Higgs bundles has a natural
formulation in terms of taking a quiver of one shape and constructing another
one using some operations available. This is reminiscent of quiver mutations
appearing in representation theory. It is natural to ask, given two quivers $I, J$
and a quiver mutation $\mu$ that takes $I$ to $J$, if we can associate a morphism of
stacks $\wh\mu : \mc{M}_{\Vect(X), I} \to \mc{M}_{\Vect(X), J}$. If quiver mutations
can be written down in terms of morphisms of simplicial sets, then we would get
a map in the opposite direction, by the fact that
$\mc{M}_{\Vect(X), -}$ is a functor
$\sSet^\op \to \St_{\sC}$ by definition. Whether or not quiver mutations can be expressed
in the language of simplicial maps, it might still be interesting to try and construct
maps associated to $\mu$ in either direction by hand.


\section{Higgs Bundles}
\label{sec:HiggsBundles}

We now discuss an application of the construction of the moduli stacks of
quiver bundles from \cref{sec:ModStQuiver}
in constructing a moduli stack parametrizing Higgs bundles and a moduli
stack parametrizing morphisms thereof --- one of the main motivations of this paper.
This is where we categorify the moduli problem for Higgs bundles in the sense
of \cref{rmk:meaning}.
We then discuss non-Abelian Hodge theory in this categorified setting.
For this section, we will only deal with $\St_{\CRing^{\heartsuit, \op}}^1$.

\subsection{Moduli Stack of Higgs Bundles}

We recall the definition
of a Higgs bundle for which we refer the reader to \cite[Example 9]{GR15} for details in a context that coincides fairly readily with our setting here. Let $X \in \St_{\CRing^{\heartsuit, \op}}^1$ and $K$, a locally free
$\mc{O}_X$--module equipped with an exterior product
$\wedge : K \otimes_{\mc{O}_X} K \to K \wedge K$.
A $K$--twisted Higgs bundle over $X$ is a pair $(E, \phi : E \to E \otimes_{\mc{O}_X} K)$
such that the following diagram commutes:
\begin{equation}\label{eqn:integrability}
\begin{tikzcd}
E \ar[r, "\phi"] \ar[dd] & E \otimes K \ar[d, "\phi \otimes \id_K"] \\
& E \otimes K \otimes K \ar[d, "\id_E \otimes \wedge"] \\
0 \ar[r] & E \otimes K \wedge K
\end{tikzcd}\end{equation}
We call the commutativity of this diagram the integrability condition.
We call $E$ the underlying bundle of the Higgs bundle and $\phi$, the Higgs field.
A morphism of $K$--twisted Higgs bundles $(E, \phi) \to (E', \phi')$ is a vector bundle map
$f : E \to E'$ making the following diagram commute:
\begin{equation}\label{eqn:Higgs-morphism}
\begin{tikzcd}
E \ar[r, "\phi"] \ar[d, "f" left] & E \otimes K \ar[d, "f \otimes \id_K"] \\
E' \ar[r, "\phi'" below] & E' \otimes K
\end{tikzcd}
\end{equation}

In other words, the data of a Higgs bundle amounts to a commutative diagram of the form
\ref{eqn:integrability}, from which we can extract the underlying bundle by taking the top-left
vertex and the Higgs field, by taking the top horizontal edge. Hence, the moduli
stack of Higgs bundles should be a substack of the moduli stack of quiver bundles of shape
\begin{equation}\label{eqn:integrability-base}
P := \begin{tikzcd}[row sep=small]
a \ar[r, "e_{ab}"] \ar[dd, "e_{ab'}" left] & b \ar[d, "e_{bc}"] \\
& c \ar[d, "e_{cd}"] \\
b' \ar[r, "e_{b'd}" below] & d
\end{tikzcd}
\end{equation}
satisfying:
\begin{enumerate}[label=(\roman*), itemsep=0pt]
\item The vertex indexed by $b$ is the one indexed by $a$ tensored with $E$.
\item The vertex indexed by $c$ is the one indexed by $b$ tensored with $K$
\item The vertex indexed by $d$ is the one indexed by $a$ tensored with $K \wedge K$
\item The vertex indexed by $b'$ is $0$.
\item The edge indexed by $e_{bc}$ is the one indexed by $e_{ab}$ tensored with
$\id_K$.
\item The edge indexed by $e_{cd}$ is the identity of vertex indexed by $a$ tensored
with $\wedge$.
\item The edges indexed by $e_{ab'}$ and $e_{b'd}$ are the unique maps.
\end{enumerate}
Notice that (iv), (v) and (vi) imply the rest.
We will use this observation to recover our desired substack as a pullback of maps out of
$\mc{M}_{\Vect(X), P}$ into some algebraic stacks,
because, that would yield algebraicity immediately, as long as the hypotheses
of \cref{thm:algebraicity} are satisfied.

There is a construction of the moduli stack of Higgs bundles given
in \cite[\S 7.4]{CW17} whose description of points partially motivates our construction.
A $U$--point of this moduli stack consists of a $pr_2^*K$--twisted Higgs bundle
$(E, \phi : E \to E \otimes_{\mc{O}_{U \times X}} pr_2^*K)$ where $pr_2 : U \times X \to X$
is the second projection --- note that pullback functors for sheaves are symmetric
monoidal, they preserve colimits as they are left adjoint to pushforward functors,
and hence, they preserve exterior algebra objects so that
the diagrams defining Higgs bundles make sense for $pr_2^*K$.
To construct our stack such that its points coincide with that of \cite[\S 7.4]{CW17},
we will need a few results to define the maps that will cut it out.

\begin{lem}\label{lem:pullback-of-maps-by-projection}
Consider the map $p_f : \mathrm{pt} \to \mc{M}_{\Vect(X), 1}$ corresponding to a triple
$(E, F, f : E \to F)$ over $X$.
Let $I \in \sSet$, $s : \Delta^1 \to I$ be a $1$--simplex and
$\mc{M}_{\Vect(X), s} : \mc{M}_{\Vect(X), I} \to \mc{M}_{\Vect(X), \Delta^1}
\simeq \mc{M}_{\Vect(X), 1}$, the map given by the functor of \cref{defn:mod-st-quiv-bun}
applied to $s$. Then, the pullback stack:
\[\begin{tikzcd}
\mc{M}_{\Vect(X), I, s} \ar[r] \ar[d] \ar[rd, phantom, "\lrcorner" very near start] &
\mc{M}_{\Vect(X), I} \ar[d, "\mc{M}_{\Vect(X), s}"] \\
\mathrm{pt} \ar[r, "p_f" below] & \mc{M}_{\Vect(X), 1}
\end{tikzcd}\]
has $U$--points corresponding to quiver bundles of shape $I$ whose edge given by $s$
is $pr_2^*f$, where $pr_2 : U \times X \to X$ is the second projection.
\end{lem}
\begin{proof}
A $U$--point of this pullback stack is a $U$--point of $\mc{M}_{\Vect(X), I}$
that corresponds to a quiver bundle whose edge given by $s$ is a $U$--point
$q : U \to \mc{M}_{\Vect(X), 1}$ that factors through $p_f$.
$p_f$ corresponds to a map
$p_f' : X \simeq \mathrm{pt} \times X \to \mc{E}_X^\vee \boxtimes \mc{E}_X$ such
that the projection to $X \times X$ factors through the diagonal
$\Delta : X \to X \times X$. Then, $q : U \to \mc{M}_{\Vect(X), 1}$ corresponds to a map
$q' : U \times X \to \mc{E}_X^\vee \boxtimes \mc{E}_X$. We consider the projections
$q'_1, q'_3 : U \times X \to \mc{M}_{\Vect(X)}$ as we did in \cref{sec:ModStVec}.
This defines defines a section of
\[
(q'_1 \times \id_X, q'_3 \times \id_X)^*(\mc{E}_X^\vee \boxtimes \mc{E}_X)
\]
which corresponds to a triple $(E', F', f' : E' \to F')$ over $U \times X$.
Under the Cartesian closed structure of stacks,
$q$ factoring through $p_f$ corresponds to $q'$ factoring through $p_f'$ as shown below:
\[\begin{tikzcd}
U \times X \ar[r, "q'"] \ar[d, "! \times \id_X" left] \ar[rd, "pr_2"] &
\mc{E}_X^\vee \boxtimes \mc{E}_X \\
\mathrm{pt} \times X \ar[r, "\simeq" below] & X \ar[u, "p_f'" right] 
\end{tikzcd}\]
That is $q'_i = (p_f' \circ pr_2)_i$ for $i = 1, 3$.
This implies that the section
that defines $(E', F', f')$ is a section of
\[
((p_f' \circ pr_2)_1 \times \id_X, (p_f' \circ pr_2)_3 \times \id_X)^*
(\mc{E}_X^\vee \boxtimes \mc{E}_X)
\]
which implies that $(E', F', f')$ is actually of the form:
\[
(pr_2^*E, pr_2^*F, pr_2^*f : pr_2^*E \to pr_2^*F)
\]
\end{proof}

\begin{guess}\label{guess:symm-mon-moduli}
There exist maps
\begin{align*}
- \otimes_0 - &: \mc{M}_{\Vect(X)} \times \mc{M}_{\Vect(X)} \to \mc{M}_{\Vect(X)} \\
- \otimes_1 - &: \mc{M}_{\Vect(X), 1} \times \mc{M}_{\Vect(X), 1} \to \mc{M}_{\Vect(X), 1}
\end{align*}
such that
\begin{enumerate}[label=(\roman*), itemsep=0pt]
\item $\otimes_0$ sends $U$--points corresponding to two vector bundles
$E, F$ over $U \times X$ to a $U$--point corresponding to the vector bundle
$E \otimes_{\mc{O}_{U \times X}} F$.

\item $\otimes_1$ sends a $U$--point corresponding
to two triples
\[
(E, F, f : E \to F), (E', F', f' : E' \to F')
\]
to a $U$--point corresponding to the triple
\[
(E \otimes_{\mc{O}_{U \times X}} F, E' \otimes_{\mc{O}_{U \times X}} F',
f \otimes_{\mc{O}_{U \times X}} f' :
    E \otimes_{\mc{O}_{U \times X}} F \to E' \otimes_{\mc{O}_{U \times X}})
 \]

\item Without making it fully precise, we also guess that these two maps
make the internal category $(\mc{M}_{\Vect(X)}, \mc{M}_{\Vect(X), 1}, s, t, e, c)$
of \cref{thm:arrow_moduli}
into an internal symmetric monoidal category in $\St_{\CRing^{\heartsuit, \op}}^1$.
For precise definitions, we refer the reader to \cite[\S 1.3]{Ghi19}. In particular,
for each $U \in \St_{\CRing^{\heartsuit, \op}}^1$, the $U$--points of this internal category
form a symmetric monoidal category.
\end{enumerate}
\end{guess}
\begin{proof}[Proof sketch]
The map $\otimes_0$ is just the tensor product functor on the groupoid
$\mc{M}_{\Vect(X)}(U)$ of vector bundles over $X \times U$, for each affine scheme $U$.
Now, the objects of $\mc{M}_{\Vect(X), 1}(U)$ are tripes $(E, F, f)$ for
vectors bundles $E, F$ and a morphism of vector bundles $f : E \to F$ over $X \times U$.
We claim that a morphism $(E, F, f) \to (E', F', f')$ in
$\mc{M}_{\Vect(X), 1}(U)$ is a pair of vector bundle maps
$a : E \to E'$, $b : F \to F'$, making the following sqaure commute:
\[\begin{tikzcd}
E \ar[r, "f"] \ar[d, "a" left] & F \ar[d, "b"] \\
E' \ar[r, "f'" below] & F'
\end{tikzcd}\]
To see this, consider the maps
$q, q' : X \times U \to \mc{E}_{X}^\vee \boxtimes \mc{E}_X$ corresponding to the triples
$(E, F, f), (E', F', f')$ and the four maps
$q_1, q_2, q_1', q_2' : X \times U \to \mc{M}_{\Vect(X)}$ obtained by composing with the
two projections $\mc{E}_X^\vee \boxtimes \mc{E}_X \to \mc{M}_{\Vect(X)}$ respectively,
and observe that any natural transformation:
\[\begin{tikzcd}
X \times U \ar[r, bend left, "q"{name=U, above}] \ar[r, bend right, "q'"{name=D, below}] &
\mc{E}_X^\vee \boxtimes \mc{E}_X
\ar[from=U, to=D, "\alpha", Rightarrow, shorten <= 0.5em, shorten >= 0.5em]
\end{tikzcd}\]
gives two natural transformations
$\alpha_1 : q_1 \To q_1', \alpha_2 : q_2 \To q_2'$. This gives two vector bundle maps
$a : E \to E', b : F \to F'$. One then has to verify that the maps $f, f', a, b$ form
a commutative square as above. With this, the map
$\otimes_1 : \mc{M}_{\Vect(X), 1}(U) \times \mc{M}_{\Vect(X), 1} \to \mc{M}_{\Vect(X), 1}$
is given by:
\[\begin{tikzcd}
((E, F, f), (E', F', f')) \mapsto (E \otimes E, F \otimes F', f \otimes f')
\end{tikzcd}\]
\[\left(\begin{tikzcd}
E \ar[r, "f"] \ar[d, "a" left] & F \ar[d, "b"] \\
E' \ar[r, "f'" below] & F'
\end{tikzcd}, \begin{tikzcd}
V \ar[r, "g"] \ar[d, "c" left] & W \ar[d, "d"] \\
V' \ar[r, "g'" below] & V'
\end{tikzcd}\right) \mapsto
\begin{tikzcd}[column sep=large]
E \otimes V \ar[r, "f \otimes g"] \ar[d, "a \otimes c" left] &
F \otimes W \ar[d, "b \otimes d"] \\
E' \otimes V' \ar[r, "f' \otimes g'" below] & F' \otimes W'
\end{tikzcd}\]

It remains to describe what $\otimes_1$ does to morphisms in $\mc{M}_{\Vect(X), 1}$
that do not lie above identity morphisms of affine schemes. In this case, consider
a morphism $u : U \to V$ of affine schemes. A morphism above this in
$\mc{M}_{\Vect(X)}, 1$ is just a pair of squares as above but where the bottom arrow
is a morphism of vector bundles above $V$ pulled back to $U$ via $\id_X \times u$.
Since the pullback functor
$(\id_X \times u)^* : \Vect(X \times V) \to \Vect(X \times U)$
preserves tensor products, we can just take:

\begin{footnotesize}
\[\left(\begin{tikzcd}[column sep=huge]
E \ar[r, "f"] \ar[d, "a" left] & F \ar[d, "b"] \\
(\id_X \times u)^*E' \ar[r, "(\id_X \times u)^*f'" below] & (\id_X \times u)^*F'
\end{tikzcd}, \begin{tikzcd}[column sep=huge]
V \ar[r, "g"] \ar[d, "c" left] & W \ar[d, "d"] \\
(\id_X \times u)^*V' \ar[r, "(\id_X \times u)^*g'" below] & (\id_X \times u)^*V'
\end{tikzcd}\right)\]
\[ \mapsto
\begin{tikzcd}[column sep=8.5em]
E \otimes V \ar[r, "f \otimes g"] \ar[d, "(\id_X \times u)^*(a \otimes c)" left] &
F \otimes W \ar[d, "(\id_X \times u)^*(b \otimes d)"] \\
(\id_X \times u)^*(E' \otimes V') \ar[r, "(\id_X \times u)^*(f' \otimes g')" below] &
(\id_X \times u)^*(F' \otimes W')
\end{tikzcd}\]
\end{footnotesize}
\end{proof}

We are now ready to define, or perhaps redefine in our language,
the moduli stacks of Higgs bundles.

\begin{defn}[Moduli Stack of Higgs Bundles]\label{defn:Higgs-moduli}
We consider again the simplicial set $P$ shown in \ref{eqn:integrability-base} and define
the following maps:
\begin{enumerate}[label=(\roman*), itemsep=0pt]
\item We will write $\mc{M}_{P}$ for $\mc{M}_{\Vect(X), P}$ and $\mc{M}_{k}$ for
$\mc{M}_{\Vect(X), \Delta^k}$, for brevity.
\item $\wh{\id_K} : \mathrm{pt} \to \mc{M}_{1}$ corresonding to the map
$\id_K : K \to K$
\item $\wh{\wedge} : \mathrm{pt} \to \mc{M}_{1}$ corresponding to the map
$\wedge : K \otimes K \to K \wedge K$
\item $z : \mathrm{pt} \to \mc{M}_{0}$ corresponding to the zero vector bundle $0$ over $X$.
\item For a simplex $s : \Delta_k \to P$, we will write $\mc{M}_{s}$ for a copy of
$\mc{M}_{\Vect(X), \Delta^k}$.
We will write $\pi_s$ for the map
$\mc{M}_{\Vect(X), s} : \mc{M}_{P} \to \mc{M}_{s}$ of \cref{defn:mod-st-quiv-bun}.
For instance, $\pi_{e_{ab}}$ is the map that sends
a quiver bundle of shape $P$ to the edge indexed by $e_{ab}$.
\end{enumerate}

With these, we first define a map:

\[\begin{tikzcd}
\mc{M}_P \ar[d, "{(\pi_{e_{ab}}, \pi_a)}"] \\
\mc{M}_{e_{ab}} \times \mc{M}_{a} \ar[d, "\id \times e"] \\
\mc{M}_{e_{ab}} \times \mc{M}_{1} \ar[d, "\simeq"] \\
\mc{M}_{e_{ab}} \times \mathrm{pt} \times \mc{M}_{1} \times \mathrm{pt} \times \mathrm{pt}
    \ar[d, "\id \times \wh{\id_K} \times \id \times \wh{\wedge} \times z"] \\
\mc{M}_{e_{ab}} \times \mc{M}_1 \times \mc{M}_{1} \times \mc{M}_1 \times \mc{M}_{b'}
    \ar[d, "\otimes_1 \times \otimes_1 \times \id"] \\
\mc{M}_{e_{bc}} \times \mc{M}_{e_{cd}} \times \mc{M}_{b'}
\end{tikzcd}\]
and call it $\eta$. We then define the moduli stack of $K$--twisted Higgs bundles
on $X$ to be the following equalizer of stacks (recall that all limits are taken
within the respective quasicategories and so in this case, this is a (2, 1)-limit):
\[\begin{tikzcd}
\mc{M}_{\mathrm{Higgs}(X, K)} \ar[r] &
\mc{M}_{P}
    \ar[d, "{(\pi_{e_{bc}}, \pi_{e_{cd}}, \pi_{b'})}" shift left=5pt]
    \ar[d, "\eta" left, shift right=5pt]\\ &
\mc{M}_{e_{bc}} \times \mc{M}_{e_{cd}} \times \mc{M}_{b'}
\end{tikzcd}\]
\end{defn}

\begin{rmk}
Using \cref{lem:pullback-of-maps-by-projection} and \cref{guess:symm-mon-moduli}, we see
that the map $\eta$ in the above construction implements the following idea:
\begin{enumerate}[label=(\roman*), itemsep=0pt]
\item Take a diagram of vector bundles over $U \times X$ of shape $P$.
\item Take the edge indexed by $e_{ab}$ and tensor it with the identity of $pr_2^*K$
\item Take the vertex indexed by $a$, take its identity morphism and tensor that
with
\[
pr_2^*\wedge : pr_2^*(K \otimes K) \simeq pr_2^*K \otimes pr_2^*K
    \to pr_2^*(K \wedge K) \simeq pr_2^*K \wedge pr_2^*K
\]
\end{enumerate}
The equalizer $\mc{M}_{\mathrm{Higgs}(X, K)}$ is then the substack of $\mc{M}_P$
whose $U$--points are diagrams of shape $P$
such that the morphism indexed by $e_{bc}$ is the morphism indexed by $e_{ab}$ tensored
with $\id_{pr_2^*K}$, and the morphism indexed by $e_{cd}$ is the identity of the vertex
indexed by $a$ tensored with $pr_2^*\wedge$, and in addition, the vertex indexed by
$b'$ is the zero vector bundle.
\end{rmk}

By this remark, we have:

\begin{thm}\label{thm:MHiggs-point}
A $U$--point of $\mc{M}_{\mathrm{Higgs}(X, K)}$ corresonds to
a commutative diagram of the form:
\[\begin{tikzcd}
E \ar[r, "\phi"] \ar[dd] & E \otimes pr_2^*K \ar[d, "\phi \otimes \id_{pr_2^*K}"] \\
& E \otimes pr_2^*K \otimes pr_2^*K \ar[d, "\id_E \otimes pr_2^*\wedge"] \\
0 \ar[r] & E \otimes pr_2^*K \wedge pr_2^*K
\end{tikzcd}\]
\end{thm}

\begin{rmk}
One point to note is that in the infinity categorical setting, this construction will lead
to diagrams that commute up to ``homotopy'' with respect to the appropriate notion of
homotopy in the given infinity categorical setting, and in this case we do not just
get a Higgs bundle but a Higgs bundle along with a specific witness for
its integrability.
\end{rmk}

\begin{rmk}
If we restrict attention to a curve $X$ over $\bb{C}$, then the commutativity of the diagram
in the above theorem is automatic as
$pr_2^*K \wedge pr_2^*K \simeq pr_2^*(K^{\wedge 2}) \simeq pr_2^*(0) = 0$. Hence, in this
case, we recover the construction of \cite[\S 7.4]{CW17}.
\end{rmk}

Up to now, we have a reconstruction of the moduli stack of Higgs bundles, but this
perspective leads to the construction of a different moduli stack that, to the best of our
knowledge, has not yet appeared in the literature --- a moduli stack whose
points are morphisms of Higgs bundles.

\begin{defn}[Moduli Stack of Higgs Bundle Morphisms]
Let us draw $S := \Delta^1 \times \Delta^1$ as:
\[\begin{tikzcd}
a \ar[r, "u"] \ar[d, "v" left] & b \ar[d, "w"] \\
c \ar[r, "y" below] & d
\end{tikzcd}\]
Consider the moduli stack $\mc{M}_{\Vect(X), S}$ which we write as $\mc{M}_{S}$. Using
notation similar to the \cref{defn:Higgs-moduli}, we consider the map:
\[
\xi :
\mc{M}_S \to[\pi_v]
\mc{M}_1 \to[\simeq]
\mc{M}_1 \times \mathrm{pt} \to[\id \times \wh{\id_K}]
\mc{M}_1 \times \mc{M}_1 \to[\otimes_1]
\mc{M}_1
\]
We define the moduli of stack of morphisms of $K$--twisted pre--Higgs bundles
as the following equalizer of stacks:
\[\begin{tikzcd}
\mc{M}_{\mathrm{preHiggs}(X, K), 1} \ar[r] &
\mc{M}_S \ar[r, shift left, "\xi"] \ar[r, shift right, "\pi_w" below] &
\mc{M}_1
\end{tikzcd}\]
Using \cref{lem:pullback-of-maps-by-projection},
we can see that the $U$--points of $\mc{M}_{\mathrm{preHiggs}(X, K), 1}$ are precisely
diagrams of vector bundles of the form:
\begin{equation}\label{eqn:preHiggs-morphism}
\begin{tikzcd}
E \ar[r, "\phi"] \ar[d, "f" left] & E \otimes pr_2^*K \ar[d, "f \otimes \id"] \\
F \ar[r, "\psi" below] & F \otimes pr_2^*K
\end{tikzcd}
\end{equation}
We write:
\[
\pi_q' : \mc{M}_{\mathrm{preHiggs}(X, K), 1} \to \mc{M}_S \to[\pi_q] \mc{M}_{\Delta^1}
\]
for $q = u : \Delta^1 \to S$ or $q = y : \Delta^1 \to S$ and
\[
\pi_{e_{ab}}' : \mc{M}_{\mathrm{Higgs}(X, K)} \to \mc{M}_{P} \to[\pi_{e_{ab}}] \mc{M}_1
\]
We then define the moduli stack of morphisms of $K$--twisted Higgs bundles to be the cone
point of the following limit diagram:
\[\begin{tikzcd}[column sep=tiny]
& & \mc{M}_{\mathrm{Higgs}(X, K), 1} \ar[lld, "s" above left] \ar[d] \ar[rrd, "t"] & &\\
\mc{M}_{\mathrm{Higgs}(X, K)} \ar[rd, "\pi_{e_{ab}}'" below left] & &
\mc{M}_{\mathrm{preHiggs}(X, K), 1} \ar[ld, "\pi_u'"] \ar[rd, "\pi_y'" below left] & &
\mc{M}_{\mathrm{Higgs}(X, K)} \ar[ld, "\pi_{e_{ab}}'"] \\
& \mc{M}_1 & & \mc{M}_1 &
\end{tikzcd}\]
We call the canonical maps $s$ and $t$ to $\mc{M}_{\mathrm{Higgs}(X, K)}$ the
source and target maps.
\end{defn}

Using arguments similar to the ones we used to examine $U$--points of the other stacks
constructed so far --- for instance, \cref{defn:Higgs-moduli} --- we have the following
result:

\begin{thm}
A $U$--point of $\mc{M}_{\mathrm{preHiggs}(X, K), 1}$
is a diagram formed by gluing the following diagrams along the $\phi$ and $\psi$ edges:
\[
\begin{tikzcd}
E \ar[r, "\phi"] \ar[dd] & E \otimes pr_2^*K \ar[d, "\phi \otimes \id_{pr_2^*K}"] \\
& E \otimes pr_2^*K \otimes pr_2^*K \ar[d, "\id_E \otimes pr_2^*\wedge"] \\
0 \ar[r] & E \otimes pr_2^*K \wedge pr_2^*K
\end{tikzcd}
\hspace{5em}
\begin{tikzcd}
F \ar[r, "\psi"] \ar[dd] & F \otimes pr_2^*K \ar[d, "\psi \otimes \id_{pr_2^*K}"] \\
& F \otimes pr_2^*K \otimes pr_2^*K \ar[d, "\id_F \otimes pr_2^*\wedge"] \\
0 \ar[r] & F \otimes pr_2^*K \wedge pr_2^*K
\end{tikzcd}
\]
\[\begin{tikzcd}
E \ar[r, "\phi"] \ar[d, "f" left] & E \otimes pr_2^*K \ar[d, "f \otimes \id"] \\
F \ar[r, "\psi" below] & F \otimes pr_2^*K
\end{tikzcd}\]
The map $s$ sends such a $U$--point to the top-left diagram
and $t$ sends it to the top-right diagram.
\end{thm}

We end this section by stating an easy corollary of our work so far and a guess.

\begin{cor}\label{cor:Higgs-algebraicity}
If $X$ satisfies the hypotheses of \cref{thm:algebraicity} and \cref{guess:symm-mon-moduli}
holds, the stacks
$\mc{M}_{\mathrm{Higgs}(X, K)}$ and $\mc{M}_{\mathrm{Higgs}(X, K), 1}$ are Artin.
\end{cor}
\begin{proof}
It suffices to see that they are constructed as iterated pullbacks of Artin stacks.
\end{proof}

\begin{guess}\label{guess:Higgs-moduli-category}
If \cref{guess:symm-mon-moduli} holds, then there exist maps
\[
e : \mc{M}_{\mathrm{Higgs}(X, K)} \to \mc{M}_{\mathrm{Higgs}(X, K), 1}
\]
and
\[
c : \mc{M}_{\mathrm{Higgs}(X, K), 1} \times_{\mc{M}_{\mathrm{Higgs}(X, K)}}
    \mc{M}_{\mathrm{Higgs}(X, K), 1} \to \mc{M}_{\mathrm{Higgs}(X, K), 1}
\]
such that $(\mc{M}_{\mathrm{Higgs}(X, K)}, \mc{M}_{\mathrm{Higgs}(X, K), 1}, s, t, e, c)$
is an internal category in $\St_{\CRing^{\heartsuit, \op}}^1$.
\end{guess}
\begin{proof}[Proof sketch]
We need only provide the maps $e$ and $c$, while the rest needs to be verified by chasing
$U$--points. The map $e$ is given by the diagram of the form \ref{eqn:preHiggs-morphism}
while the composition map is given by pasting such diagrams and taking the composition
of the composeable arrows.
\end{proof}

\subsection{Non-Abelian Hodge Theory}

In line with the categorification of the moduli problem for Higgs bundles, it is natural
to consider the categorified versions of the Hitchin morphism and the non-Abelian Hodge
correspondence. The questions that need to be answered in order to make these ideas
precise are as follows:

\begin{enumerate}
\item Denoting the Hitchin base as $H$, we can then ask:
is there a stack $H_1$ and maps $s, t : H_1 \to H$, $e : H \to H_1$ and
$c : H_1 \times_H H_1 \to H_1$ that form an internal category in stacks?
See the Conventions subsection in \cref{subsubsec:Conventions} for
the interpretations of $s, t, e, c$ and \cite{Ghi19} for detailed discussion
of internal category theory.

\item The Hitchin morphism $\mathrm{Hitchin} : \mc{M}_{\mathrm{Higgs}(X, K)} \to H$
is given by viewing a Higgs field $\phi : E \to E \otimes K$
as a global section of $E^\vee \otimes E \otimes K$ and taking the maps
$c_i : E^\vee \otimes E \otimes K \to K^{\otimes i}$ given by the
$i$--th characteristic coefficient maps. We can then ask: is there a morphism
\[
\mathrm{Hitchin}_1 : \mc{M}_{\mathrm{Higgs}(X, K), 1} \to H_1
\]
that together with the usual Hitchin morphism,
form an internal functor of category stacks? That is, are there commutative diagrams
of the form:
\[\begin{tikzcd}[row sep=huge, column sep=huge]
\mc{M}_{\mathrm{Higgs}(X, K), 1} \ar[r, "\mathrm{Hitchin}_1"]
  \ar[d, shift right=11pt, "s" description] \ar[d, shift right=5pt, "t" description] &
H_1
  \ar[d, shift right=7pt, "s" description] \ar[d, shift right=1pt, "t" description] \\
\mc{M}_{\mathrm{Higgs}(X, K)} \ar[r, "\mathrm{Hitchin}" below]
    \ar[u, shift right=1pt, "e" description] &
H \ar[u, shift right=5pt, "e" description]
\end{tikzcd}\]
\[\begin{tikzcd}[row sep=huge, column sep=huge]
\mc{M}_{\mathrm{Higgs}(X, K), 1}
\times_{\mc{M}_{\mathrm{Higgs}(X, K)}} \mc{M}_{\mathrm{Higgs}(X, K), 1}
\ar[r, "\mathrm{Hitchin}_1 \times_{\mc{M}_{\mathrm{Higgs}(X, K)}} \mathrm{Hitchin}_1"]
  \ar[d, "c" left] &
H_1 \times_{H} H_1 \ar[d, "c"] \\
\mc{M}_{\mathrm{Higgs}(X, K), 1} \ar[r, "\mathrm{Hitchin}_1" below] &
H_1
\end{tikzcd}\]

\item What is the category stack that parametrizes
representations of the fundamental group?
That is, are there moduli stacks $\mc{M}_{\mathrm{Rep}(\pi_1(X))}$ and
$\mc{M}_{\mathrm{Rep}(\pi_1(X)), 1}$ such that the points of the first correspond to
representations of some sort of fundamental group of $X$ and those of the second
correspond to morphisms of such representations in any reasonble sense?
Do they form a category stack? Does it admit a functor of internal
categories to the categorified Hitchin base as we have asked for Higgs bundles?
There is a potential answer to the first few questions: for any stack $X$, we can
take its Betti ``shape'' $X_{\mathrm{Betti}}$ --- described in, for example,
\cite{Por17} but originating in works of Carlos Simpson --- and
consider $\mc{M}_{\Vect(X_{\mathrm{Betti}}), \Delta^n}$ for $n = 0, 1, \dots$.
For $n= 0$, the points of this stack are representations of $\pi_1(X)$, when $X$ is
a smooth complex projective curve, say. So, it is natural to expect that points of
the morphism stack, that is
$\mc{M}_{\Vect(X_{\mathrm{Betti}}), 1} = \mc{M}_{\Vect(X_{\mathrm{Betti}}), \Delta^1}$,
will correspond to morphisms of $\pi_1(X)$--representations.

\item Similarly, what is the right definition of a category stack paramentrizing
flat connections and their morphisms? Again, there is an answer in terms of shapes:
take $\mc{M}_{\Vect(X_{\mathrm{dR}}), \Delta^n}$ for $n = 0, 1, \dots$, where
$X_{\mathrm{dR}}$ is the de Rham shape of $X$, also described in \cite{Por17}.

\item Can the equivalences of moduli spaces taking part in the Riemann-Hilbert
correspondence and the non-Abelian Hodge correspondence be extended to equivalences
of the respective internal category stacks being asked for in the previous questions?
There is a pressing issue to be addressed here: there is no known morphism of stacks
extending the homeomorphism of moduli spaces realizing the
non-Abelian Hodge correspondence.
Hence, on the nose, this cannot be done --- even at the object level,
there is no known map.

However, we can consider the following strategy. There is a dg--equivalence of categories
between Higgs bundles and flat connections, for a suitably chosen base. One can try to
assemble the mapping of morphisms taking part in this equivalence into a mapping
of stacks:
\[
\mc{M}_{\mathrm{Higgs}(X, K), 1} \longleftrightarrow \mc{M}_{\Vect(X_{\mathrm{dR}}), 1}
\]
Then, one can take the composite:
\[
\mc{M}_{\mathrm{Higgs}(X, K)} \to[e]
\mc{M}_{\mathrm{Higgs}(X, K), 1} \to
\mc{M}_{\Vect(X_{\mathrm{dR}}), 1} \to
\mc{M}_{\Vect(X_{\mathrm{dR}})}
\]
where the first map sends an object to its identity arrow and the
last map sends an arrow to its source or target. This will be one possible mapping
of stacks but it is most likely not the non-Abelian Hodge map away from the
stable locus. However, one can then consider properties of this map to measure
how badly the non-Abelian Hodge map fails away from the stable locus.
\end{enumerate}

We can ask even more questions analogous to questions explored in the usual study of
Higgs bundles. For instance,
If the answer to (2) above is yes, what do the fibres of
$\mathrm{Hitchin}_1$ look like? Whatever the fibres are, for each fibre, we can consider
the image of the fibre under the source and target maps
$\mc{M}_{\mathrm{Higgs}(X, K), 1} \to \mc{M}_{\mathrm{Higgs}(X, K)}$. That is, for each
fibre $Y$ of $\mathrm{Hitchin}_1$ we get two fibres $Y_0, Y_1$ of
$\mathrm{Hitchin}$ and pair of maps:
\[
Y_0 \ot[s] Y \to[t] Y_1
\]
We can then ask what kind of information $Y$ might contain about the $Y_0$ and $Y_1$ and
if $Y$ somehow relates these two fibres. We can go further and ask what the analogue
of the $\bb{C}^*$--action is in this setting and so on.

Finally, we note that what we have discussed in this paper is just one level of
categorification. That is we have only discussed internal categories in stacks
such that its object of objects parametrizes vector bundles, Higgs bundles, quiver
bundles, etc. and its object of morphisms parametrizes morphisms of these objects.
We can ask for simplicial stacks whose $n$--th level parametrizes $n$--simplices
in the quasicategories of vector bundles, quiver bundles, Higgs bundles in the derived
algebraic, spectral algebraic and derived analytic settings.
Then, we can ask for a non-Abelian Hodge theory in this more
complete setting. The considerations discussed in this subsection
will form a major portion of the first named author's upcoming thesis.

\section{Moduli Theory and Homotopy Theory}
\label{sec:ModHmtpy}

We conclude by discussing some possible links of the perspectives of this paper with
homotopy theory.

\subsection{General Moduli Theory}

Thinking of $1$--stacks as categories fibred in groupoids over
the $1$--category of affine schemes, there is a more intuitive definition
of stacks of quiver bundles. Consider a $1$--category $I$ and consider
quiver bundles of shape $I$ over $X$.
Now, take the prestack $pr_2 : I \times \s{C} \to \s{C}$, where the product
is just the product of categories. We can then take the mapping prestack
$\Map(I \times \s{C}, \mc{M}_{\Vect(X)})$ whose
$U$--points are maps $U \times I \times \s{C} \to \mc{M}_{\Vect(X)}$
which are, in turn, equivalent to maps
$U \times I \times \s{C} \times X \to \Vect$. However, these are again equivalent
to maps $I \times \s{C} \to \Vect(U \times X)$, which commute with the projections
to $\s{C}$. Such a map is just a diagram of vector bundles over $U \times X$.
We note that the mapping prestack is, in fact, a stack since
$\mc{M}_{\Vect(X)}$ is. Hence, this serves as an alternate definition of a
moduli stack of quiver bundles over $X$. If we considered $1$--stacks as just
sheaves on $\s{C}$, then this would be equivalent to the mapping stack
$\Map(\underline{I}, \mc{M}_{\Vect(X)})$, where $\underline{I}$ is the constant
sheaf valued at $I$. With this, we would like to make a conjecture:

\begin{defn}
Let $X \in \St_\s{C}$ be a stack over $\s{C}$.
We define
\begin{align*}
\mc{M}_{\Vect(X), I}^{cat}
:=& \Map\br{\underline{I}, \mc{M}_{\Vect(X)}} \\
\simeq& \Map\br{\underline{I}, \Map(X, \Vect)} \\
\simeq& \Map(X \times \underline{I}, \Vect)
\end{align*}
to be the category theoretic
moduli stack of quiver bundles over $X$.
\end{defn}

\begin{cnj}
There is an equivalence of stacks
$\mc{M}_{\Vect(X), I}^{cat} \simeq \mc{M}_{\Vect(X), I}$ for all $I$. In fact, this equivalence is a natural equivalence of functors
$\sSet \to \St_{\CRing^{\heartsuit, \op}}^1$.
\end{cnj}

\begin{rmk}
For another definition of the above concept,
we could consider the functor category
$\Fun\br{I, \unstr\br{\mc{M}_{\Vect(X)}}}$
but we need a Cartesian fibration from it to $\s{C}$ in order to consider
it as even a prestack. We could take the map
\[
\Fun\br{I, \unstr\br{\mc{M}_{\Vect(X)}}}
\to \Fun\br{I, \s{C}}
\to \s{C}
\]
where the first map is given by post-composition by the Cartesian fibration
$\unstr\br{\mc{M}_{\Vect(X)}} \to \s{C}$ and the second map is taking
either a colimit or a limit. However, one needs to show that this
will give a Cartesian fibration and will straighten to a sheaf, and then
show an appropriate algraicity in the concrete examples we have in mind.
We can then try to show an equivalence with the above definition.
\end{rmk}

While we did not pursue this perspective in this paper, it is worth noting
that this is a more general approach as it does not require the universal
vector bundle $\mc{E}_X$. Recall that the most crucial object in our constructions
was the vector bundle $\mc{E}_X^\vee \boxtimes \mc{E}_X$.
That is, for a general moduli problem
$\mc{M} : \mc{C}^\op \to \s{S}$ where there is no analogue of
$\mc{E}_X^\vee \boxtimes \mc{E}_X$, we can still define moduli stacks
$\mc{M}_{I}$ whose points are $I$--shaped diagrams of objects parametrized
by the original moduli problem as simply the mapping prestack
$\Map(\underline{I}, \mc{M})$, which will, at the very least, be a stack
as long as $\mc{M}$ is.
An interesting example of this would be
the moduli problem for elliptic curves.
With this setup, one may examine the conditions on $\mc{M}$ needed to make $\mc{M}_I$ algebraic or have other desirable properties.
We intend to pursue this more general approach for a wide range of moduli
problems in future work.

\subsection{Speculations around Homotopy Theory}

We end this section by discussing, in imprecise terms,
some potential links between moduli theory and homotopy theory facilitated
by the perspectives of this paper.
Homotopy theory, in the most elementary sense,
is concerned with studying mapping spaces between topological spaces.
That is, the sets of maps between topological spaces are not just discrete
sets and come equipped with their own topology. Said differently, if
we restrict attention to nice enough --- that is, compactly generated weakly
Hausdorff (CGWH) topological spaces --- then the category of such spaces
admit an enrichment in Kan complexes, which gives a homotopy theory that
is equivalent to the homotopy theory of simplicial sets. There are two ways of saying
this more precisely:
\begin{enumerate}
\item The category of CGWH spaces admits
a model structure Quillen equivalent to the Kan-Quillen model structure
of simplicial sets.
\item The simplicial or coherent nerve of the simplicial category of CGWH
spaces is equivalent as a quasicategory to the quasicategory of Kan complexes.
\end{enumerate}

Now, given a classical moduli problem valued in groupoids
$\mc{M} : \s{C}^\op \to \Grpd$, we can take the
moduli stacks $\mc{M}_{\Delta^n}$ for $n = 0, 1, 2, \dots$,
as defined in the previous subsection along with the natural maps
\[
\mc{M}_{f} : \mc{M}_{\Delta^n} = \Map(\Delta^n, \mc{M})
\to \Map(\Delta^m, \mc{M}) = \mc{M}_{\Delta^m}
\]
induced by maps $f : \Delta^m \to \Delta^n$.
This data assembles into a simplicial object
$\mc{M}_\bullet = \mc{M}_{\Delta^\bullet} : \Delta^\op \to \St_\s{C}^1$.

On the other hand, for any two objects of $\mc{M}$ given by maps
$a : U \to \mc{M}$ and $b : V \to \mc{M}$, inspired by \cite[\S 1.2.2]{HTT},
we can define a simplicial stack $\mc{M}^R(a, b)_\bullet$
whose $n$--th degree is defined by the pullback:
\[\begin{tikzcd}[row sep=huge]
\mc{M}^R(a, b)_n \ar[r] \ar[d] \ar[rd, phantom, "\lrcorner" very near start] &
\mc{M}_{\Delta^{n + 1}} \ar[d, "{(d_0, \dots, d_{n + 1})}"] \\
U \times V \ar[r, "{(a, \dots, a, b)}" below] & \mc{M}^{n + 1} \times \mc{M}
\end{tikzcd}\]
where the $d_i$ are the face maps corresponding to vertex inclusions
$\Delta^0 \to \Delta^n$.
Roughly speaking,
$\mc{M}^R(a, b)_0$ is to be thought of as the stack parametrizing
maps $a \to b$; $\mc{M}^R(a, b)_1$, as the stack parametrizing homotopies
$H : g \To h$ of maps $g, h : a \to b$;
$\mc{M}^R(a, b)_2$, as the stack parametrizing three homotopies
$H_{gh} : g \To h, H_{hr} : h \To r$ and $H_{gr} : g \To r$, and a higher homotopy
between $H_{hr} \circ H_{gh}$ and $H_{gr}$; and, so on.
Changing $(a, \dots, a, b)$ to $(a, b, \dots, b)$ above, we define
a simplicial stack $\mc{M}^L(a, b)_\bullet$. One can then try defining something
akin to an enriched category whose objects form not a set but a stack
$\mc{M}$ and for any two objects $a, b \in \mc{M}$, the simplicial stack of ``maps''
$a \to b$ is $\Map^R(a, b)_\bullet$ or $\Map^L(a, b)_\bullet$. Let us denote this
object, if it exists, as $\mc{M}^\Delta$.
The constructions $\mc{M}_\bullet$ and $\mc{M}^\Delta$ can be viewed as using
the data of a classical moduli problem to produce an infinity categorical moduli
problem that is different from taking its derived or spectral version.

These constructions are reminiscent of
the theory of simplicial sets and we can ask if any of the aspects of the homotopy
theory of simplicial sets or the theory of quasicategories can be translated to this
setting, and, if they can be related to available homotopy theories in algebraic geometry
such as \cite{MV99} or the more recent \cite{AI23}.

A few questions that one might ask in this direction are:

\begin{enumerate}
\item Taking $\mc{M} = \mc{M}_{\Vect(X)}$ and two objects
$a, b \in \mc{M}_{\Vect(X)}(U)$ corresponding to two vector bundles $E_a, E_b$ over
$U \times X$, how do the mapping stacks
$\mc{M}^R(a, b)$ compare to the Hom stack $\HHom(E_a, E_b)$ seen as an object
in the quasicategory of motivic spaces of \cite{MV99} or that of motivic spectra in
\cite{AI23}? It is likely that in the Morel-Voevodsky quasicategory,
these mapping stacks will be contractible but not so in the quasicategory of
non-$\bb{A}^1$-invraiant motivic spectra.

\item Taking $\mc{M} = \mc{M}_{\Vect(X)}$ again, we can consider the
simplicial object $\mc{M}_\bullet : \Delta^\op \to \St^1_{\s{C}}$ composed
with the inclusion $\St_\s{C}^1 \hto \St_\s{C}$ followed by one of the localization
functors of Morel-Voevodsky or Annala-Iwasa. In this situation, we have a simplicial
motivic space or a simplicial motivic spectrum and we can ask: how do the
maps relate the homotopical properties of the individual levels of the simplicial
object with one another?

\item We can ask the same questions taking $\mc{M}$ to be some moduli stack of
elliptic curves or of Higgs bundles, and so on.
\end{enumerate}

\vspace{20pt}

\printbibliography

\vspace{150pt}
~~~\\
~~~\\

\end{document}